%% file: article.tex
\documentclass[10pt,oneside,reqno]{amsart}
\title{Lugares geométricos asociados a dos puntos del plano}
\makeatletter
\g@addto@macro{\endabstract}{\@setabstract}
\newcommand{\authorfootnotes}{\renewcommand\thefootnote{\@fnsymbol\c@footnote}}%
\makeatother
\usepackage{amsmath,amssymb,amsthm,graphicx,cancel,pifont,pb-diagram,mathpazo,epstopdf,wasysym,marvosym,mathrsfs,eso-pic,wallpaper,array,arydshln,multicol,etoolbox,morefloats,etex,mathtools,wrapfig}
\usepackage[activeacute,english,greek,spanish]{babel}
\usepackage[square,numbers]{natbib}
\usepackage[labelfont=bf]{caption}
\usepackage[font={small,it}]{caption}
\usepackage[belowskip=-10pt,aboveskip=6pt,labelsep=period]{caption}
\usepackage[LGR,T1]{fontenc}
\usepackage[10pt,GlyphNames,boldLipsian]{teubner}
\usepackage[titletoc,toc,page]{appendix}
\usepackage[scale=0.9]{ccicons}
\usepackage{fonttable}
\DeclareFontFamily{U} {MnSymbolC}{}
\DeclareFontShape{U}{MnSymbolC}{m}{n}{
    <-6>  MnSymbolC5
   <6-7>  MnSymbolC6
   <7-8>  MnSymbolC7
   <8-9>  MnSymbolC8
   <9-10> MnSymbolC9
  <10-12> MnSymbolC10
  <12->   MnSymbolC12}{}
\DeclareFontShape{U}{MnSymbolC}{b}{n}{
    <-6>  MnSymbolC-Bold5
   <6-7>  MnSymbolC-Bold6
   <7-8>  MnSymbolC-Bold7
   <8-9>  MnSymbolC-Bold8
   <9-10> MnSymbolC-Bold9
  <10-12> MnSymbolC-Bold10
  <12->   MnSymbolC-Bold12}{}
\DeclareSymbolFont{MnSyC} {U} {MnSymbolC}{m}{n}
\DeclareMathSymbol{\squaredots}{\mathrel}{MnSyC}{14}

\usepackage{tikz}

\usepackage{hyperref}
\hypersetup{bookmarksopen,bookmarksnumbered,colorlinks,linkcolor=blue,urlcolor=blue,citecolor=blue}
\newtheorem{teor}{Teorema}
\newtheorem{corol}{Corolario}
\newtheorem{obser}{Observación}
\newtheorem{prop}{Proposición}

\newtheorem{ejer}{Ejercicio}

\newtheorem*{solu}{Solución}

\newtheorem{pro}{Problema}

\newtheorem{discu}{Discusión}

\begin{document}\maketitle
\begin{center}
\normalsize
\authorfootnotes
\authorfootnotes
Jaime Chica Escobar\footnote{\email{gilrupez@hotmail.com}}\textsuperscript{1}, Hernando Manuel Quinta \'{A}vila\footnote{\email{hernandoquintana@itm.edu.co}}\textsuperscript{2} and
Jonathan Taborda Hern\'{a}ndez\footnote{\email{taborda50@gmail.com}}\textsuperscript{3}, \par \bigskip
\textsuperscript{1} Profesor jubilado UdeA, Facultad de Ciencias Exactas y Naturales, Departamento de Matem\'{a}ticas, Medell\'{i}n-Colombia.\par
\textsuperscript{2} Instituto Tecnol\'{o}gico Metropolitano, ITM, Facultad de Ciencias Económicas y Administrativas, Medell\'{i}n-Colombia.\par \bigskip
\textsuperscript{3} Acad\'{e}mico independiente.\par \bigskip
\today
\end{center}
\begin{abstract}
Tanto la elipse como la hip\'{e}rbola son lugares geométricos que se pueden definir estableciendo una relación entre puntos $P$ del plano y dos puntos fijos $A$ y $B$ (que son sus focos $F'=A$ y $F=B$).\\ Dados dos puntos $A$ y $B$ del plano (que ya no llamamos los focos $F$ y $F'$), vamos a presentar tres lugares geom\'{e}tricos asociados a $A$ y $B$ distintos de la elipse y la hipérbola.
\end{abstract}
\selectlanguage{english}
\begin{abstract}
Both the ellipse and the hyperbola are geometric places that can be defined by establishing a relationship between points $P$ of the plane and two fixed points $A$ and $B$ (which are its foci $F'=A$ and $F=B$). \\ Given two points $A$ and $B$ of the plan (which we no longer call the foci $F$ and $F'$), we are going to present three geometric places associated with $A$ and $B$ other than the ellipse and the hyperbola.
\end{abstract}
\selectlanguage{spanish}
\tableofcontents
\include{Lugares}

\phantomsection
\nocite{*}
\bibliographystyle{abbrvnat}
\bibliography{hiper}

\end{document}

%% file: Lugares.tex
\section{Introducción}
Tanto la elipse como la hipérbola son lugares geométricos que se pueden definir estableciendo una relación entre puntos $P$ del plano y dos puntos fijos $A$ y $B$ (que son sus focos $F'=A$ y $F=B$).\\[.3cm]
\textbf{La elipse}: Lugar de los puntos $P$ del plano para los cuales: $$PA+PB=2a,\quad(\text{$a$ constante $2a>\overline{AB}$})\quad\text{(véase Fig. \ref{185})},$$
curva que se estudió en el texto de la \textit{La elipse}.\footnote{Cf. Chica y Quintana (\cite{escobar_tratado_2013}).}
\begin{figure}[ht!]
\begin{center}
\includegraphics[scale=0.3]{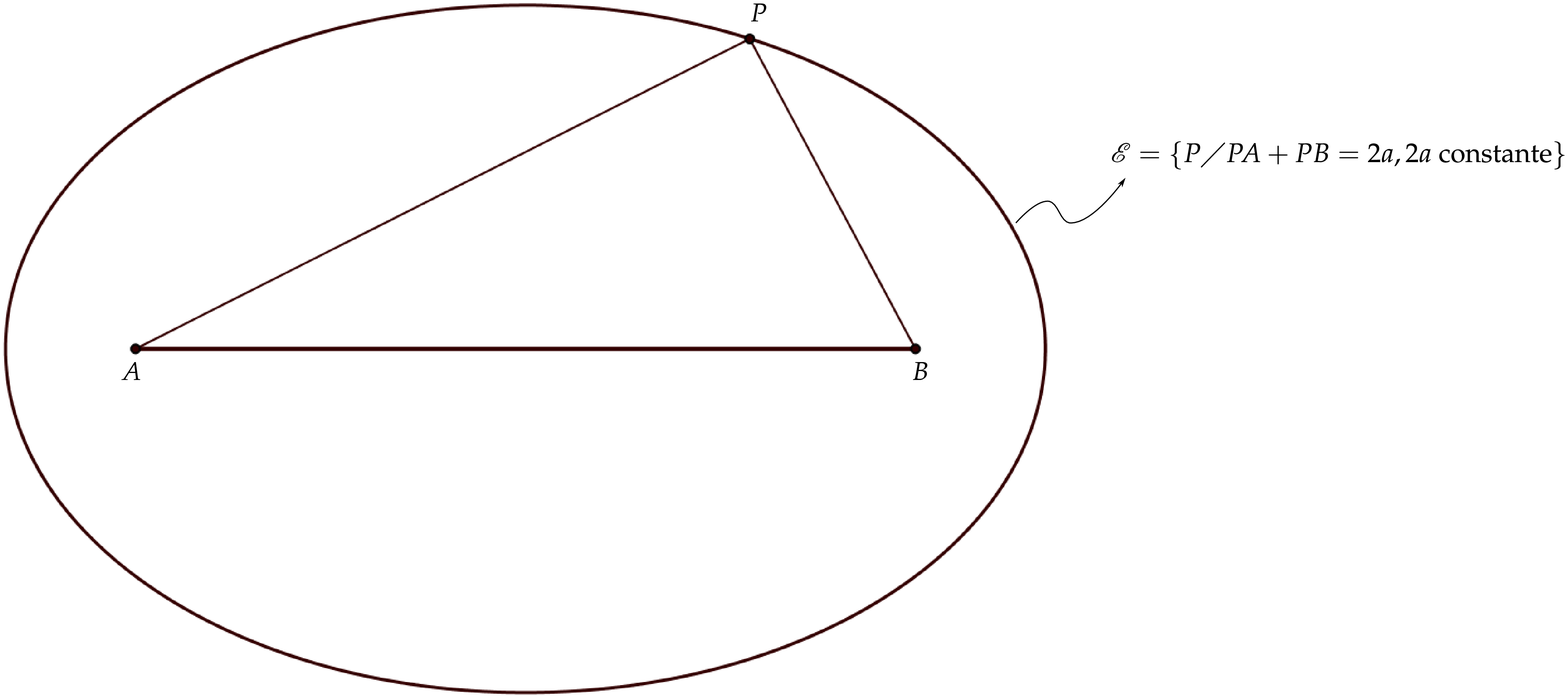}\\
\caption{La Elipse $\mathscr{E}=\left\{P\diagup PA+PB=2a, \text{2a constante}\right\}$}\label{185}
\end{center}
\end{figure}
\newline
\textbf{La hipérbola}: Lugar de los puntos $P$ del plano para los cuales: $$|PA-PB|=2a,\quad(\text{$a$ constante $2a<\overline{AB}$})\quad\text{(véase Fig. \ref{186})}$$
\begin{figure}[ht!]
\begin{center}
\includegraphics[scale=0.2]{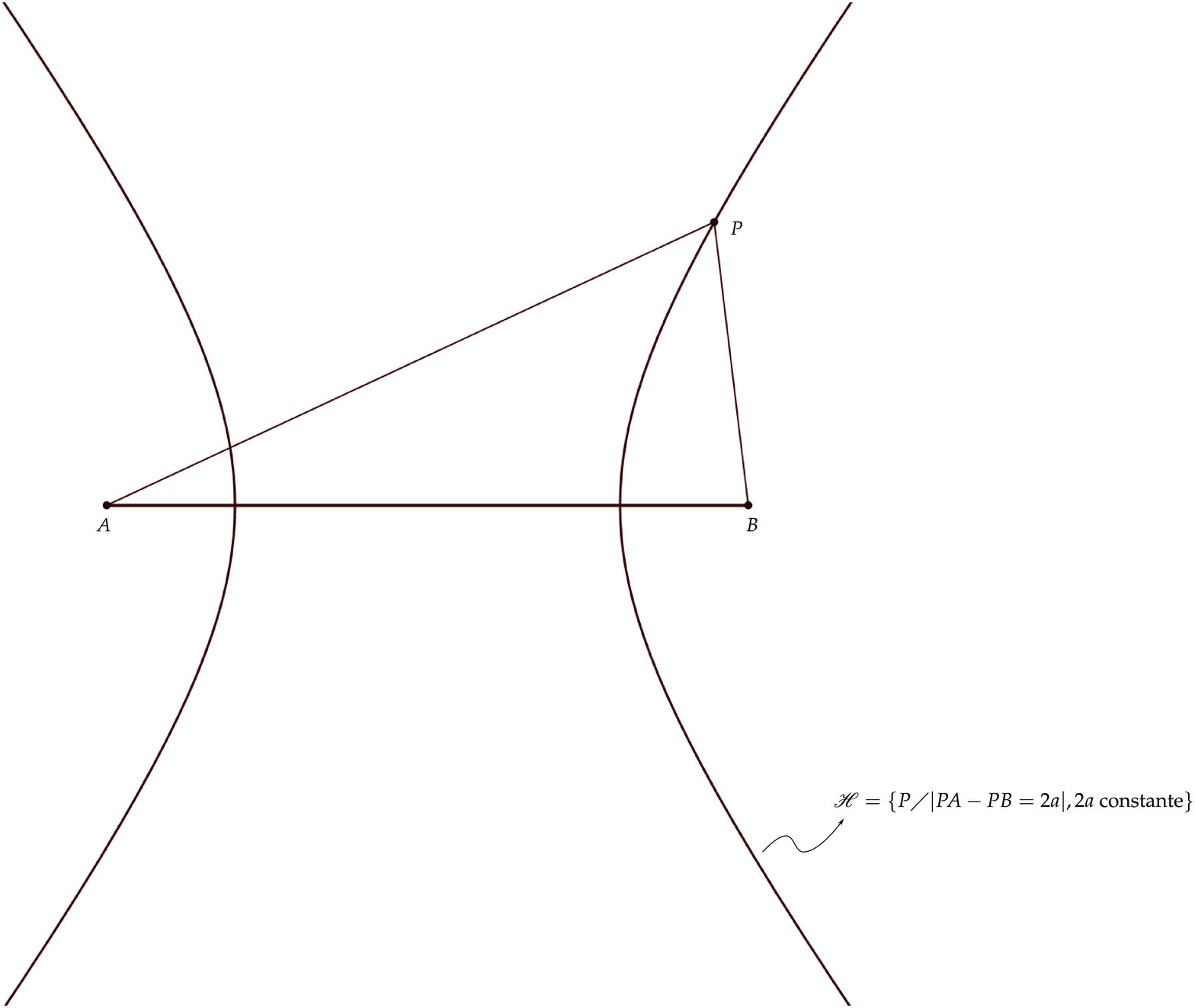}\\
\caption{La hipérbola: $\mathscr{H}=\left\{|PA-PB|=2a,\quad(\text{$a$ constante $2a<\overline{AB}$}\right\}$}\label{186}
\end{center}
\end{figure}
\newline
Curva que se ha estudiado a lo largo de la monografía sobre la \textit{Hipérbola}.\footnote{Cf. Chica y Quintana (\cite{Chica2019})}\\[.3cm]
Dados dos puntos $A$ y $B$ del plano (que ya no llamamos los focos $F'$ y $F$), vamos a presentar tres lugares geométricos asociados a $A$ y $B$ distintos de la elipse y la hipérbola.
\begin{enumerate}
\item[(1)] El lugar de los puntos $X$ del plano para los cuales $\dfrac{XA}{XB}=\text{$\lambda; \lambda\in\mathscr{R}, \lambda>0$}$.\\ En general, si $\lambda\neq 1$ este lugar es una circunferencia cuyo centro está sobre la recta $\overset{\longleftrightarrow}{AB}$ llamada la <<circunferencia de Apollonius para $A$ y $B$ y de razón $\lambda$>>, denotada por, (véase Fig. \ref{102}). $$\mathscr{C}_{AB,\lambda}:\left\{X\in\mathbb{R}^2\diagup\dfrac{XA}{XB}=\lambda, \lambda:\text{constante}\right\}$$
\begin{figure}[ht!]
\begin{center}
\includegraphics[scale=0.3]{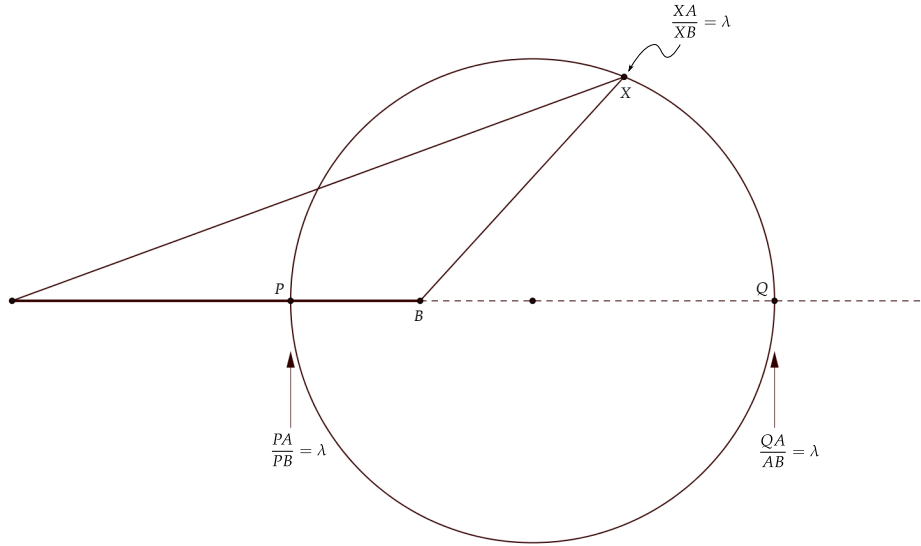}\\
\caption{Círculo de Apollonius; $\mathscr{C}_{AB,\lambda}:\left\{X\in\mathbb{R}^2\diagup\dfrac{XA}{XB}=\lambda, \lambda:\text{constante}\right\}$}\label{102}
\end{center}
\end{figure}
\newline
En el caso $\lambda=1$ el lugar geométrico descrito por $X$ es la recta mediatriz del segmento determinado por $\overline{AB}$.\\
\item[(2)] El lugar de los puntos $P$ en los que se tiene que ${PA}^2+{PB}^2=K^2:\text{constante dada.}$\\[.3cm] El lugar es ahora una circunferencia de centro en el punto medio de $\overline{AB}$, (véase Fig. \ref{194}).
\begin{figure}[ht!]
\begin{center}
\includegraphics[scale=0.3]{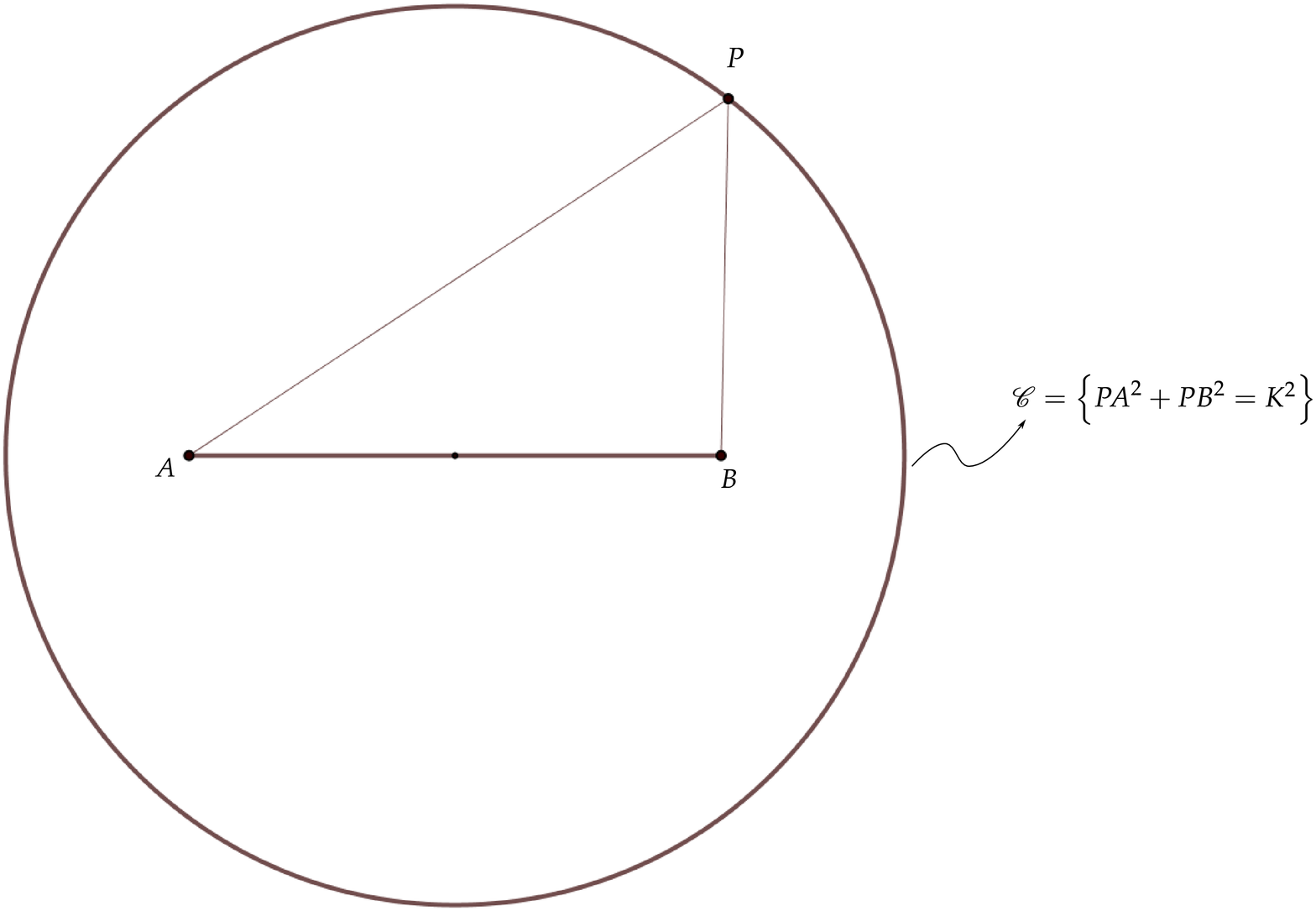}\\
\caption{$P\in\mathscr{C}\Longleftrightarrow{PA}^2+{PB}^2=K^2$}\label{194}
\end{center}
\end{figure}
\newline
\item[(3)] El lugar de los puntos $P$ del plano en los que se cumple que ${PA}^2-{PB}^2=K^2$. Ahora el lugar es una recta perpendicular a $\overline{AB}$, (véase Fig. \ref{195}).
\begin{figure}[ht!]
\begin{center}
\includegraphics[scale=0.3]{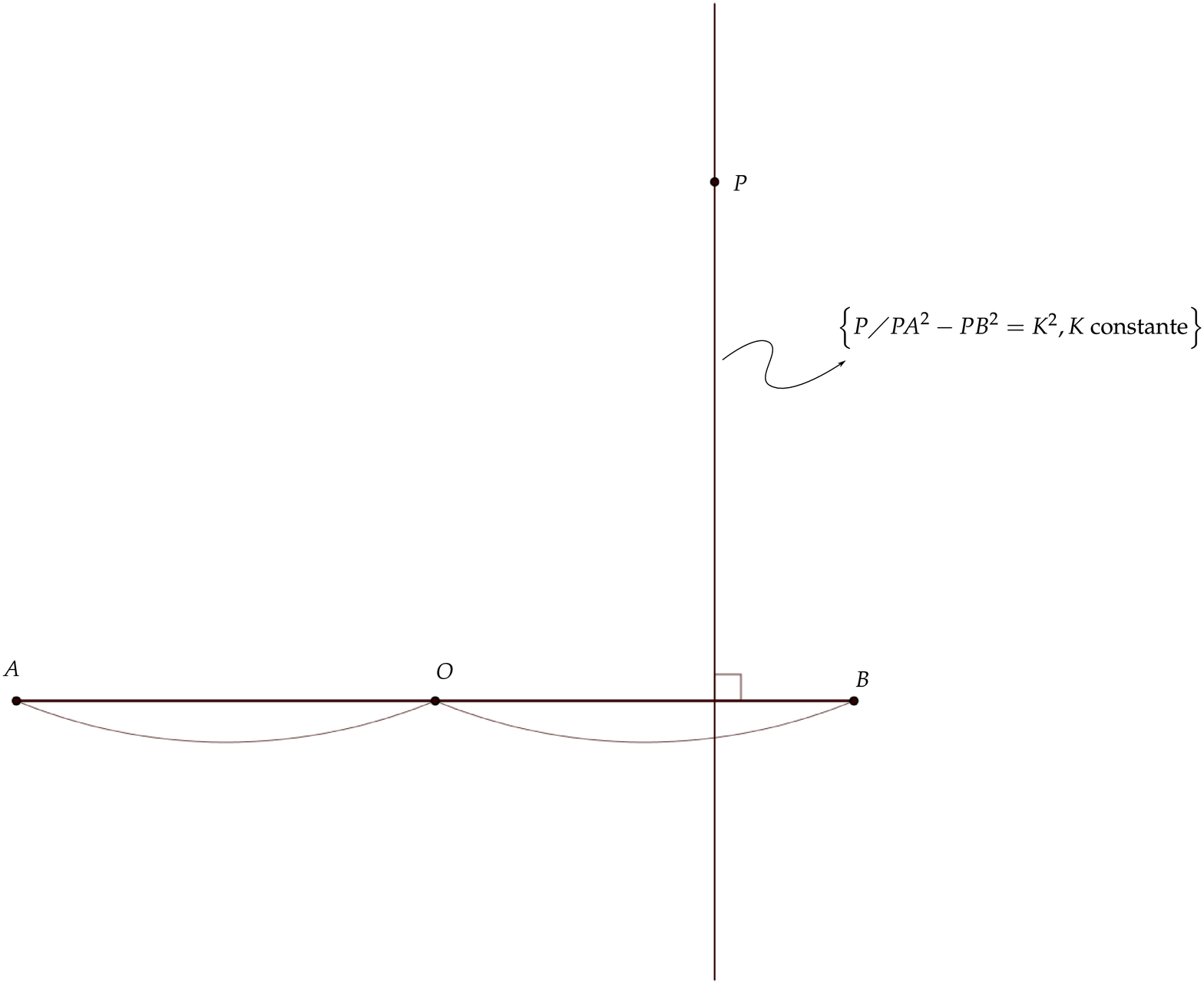}\\
\caption{$\left\{P/{PA}^2-{PB}^2=K^2, \text{K constante}\right\}$ recta perpendicular a $\overline{AB}$.}\label{195}
\end{center}
\end{figure}
\newline
A continuación se presentan por separado el estudio para cada uno de estos lugares geométricos asociados a dos puntos fijos $A$ y $B$.
\subsection{Primer lugar geométrico: círculo de Apollonius}
\noindent Para empezar iniciemos estudiando la noción de \textit{conjugados armónicos} de $A$ y $B$ para una razón $\lambda=\dfrac{m}{n}$. Fijemos $A$ y $B$ y tomemos un punto $X$ que recorre la recta $\overset{\longleftrightarrow}{AB}$, (veáse Fig. \ref{187}).
\begin{figure}[ht!]
\begin{center}
\includegraphics[scale=0.3]{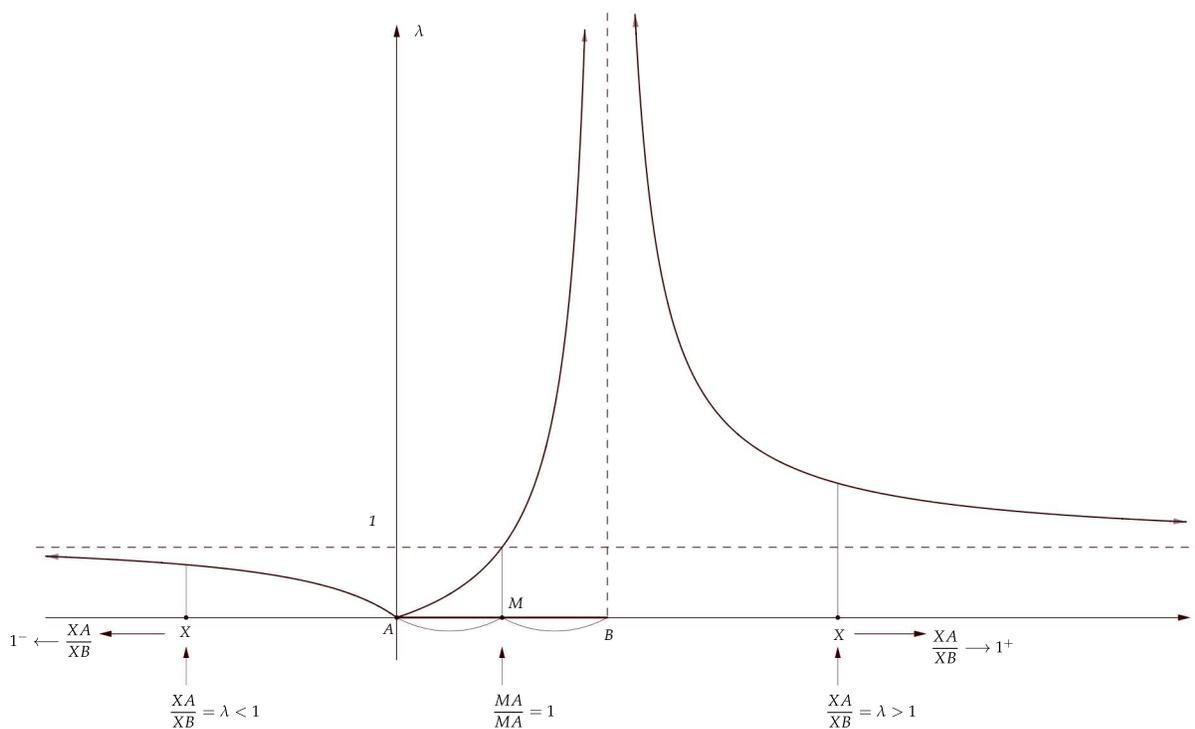}\\
\caption{Variación de $\lambda=\dfrac{XA}{XB}$ cuando $X$ se desplaza por la recta $\overset{\longleftrightarrow}{AB}$}\label{187}
\end{center}
\end{figure}
\newline
\item[$\bullet$] Si tomamos $X$ a la izquierda de $A$, $\lambda=\dfrac{XA}{XB}<1$.\\[.3cm]
\item[$\bullet$] Cuando $X$ se aleja hacia la izquierda de $A, \lambda=\dfrac{XA}{XB}\rightarrow 1^{-}$.\\
\item[$\bullet$] Cuando $X$ se tome en $A, \lambda=\dfrac{XA}{XB}=\dfrac{0}{AB}=0$.\\
\item[$\bullet$] Si $X$ se toma entre $A$ y $B, \lambda$ crece como lo indica la (Fig. \ref{187}).\\[.3cm] Si $X$ se toma en $M$, punto medio de $\overline{AB}, \lambda=\dfrac{MA}{MB}=1$ y cuando $X$ se acerque a $B, \lambda\rightarrow\empty$ ya que $XB\rightarrow 0$.\\[.3cm] Para $X$ a la derecha de $B, \lambda=\dfrac{XA}{XB}>1$ y a medida que $X$ se aleje, $\lambda$ decrece teniéndose que $\lambda\rightarrow 1$.\\[.3cm] La variación de $\lambda$ aparece indicada en la Fig. \ref{187}. De suerte que si tomamos $0<\lambda_0<1$,\\
se pueden encontrar dos puntos $P$ y $S$, $P$ entre $A$ y $M$ y $S$ a la izquierda de $A$ de modo que: $$\dfrac{SA}{SB}=\dfrac{PA}{PB}=\lambda_0,\quad 0<\lambda_0<1$$ $P$ y $S$ se llaman los \textit{conjugados armónicos de A y B para $\lambda_0$}.\\[.3cm] Si tomamos $\lambda_1>1$, podemos hallar $Q$ y $R$ en $\overset{\longleftrightarrow}{AB}$, $Q$ entre $M$ y $B, R$ a la derecha de $B$, de manera que, $$\dfrac{QA}{QB}=\dfrac{RA}{RB}=\lambda_1\quad\lambda_1>1$$ $Q$ y $R$ se llaman los \textit{conjugados armónicos de A, B para $\lambda_1$}, (véase Fig. \ref{188}).
\begin{figure}[ht!]
\begin{center}
\includegraphics[scale=0.3]{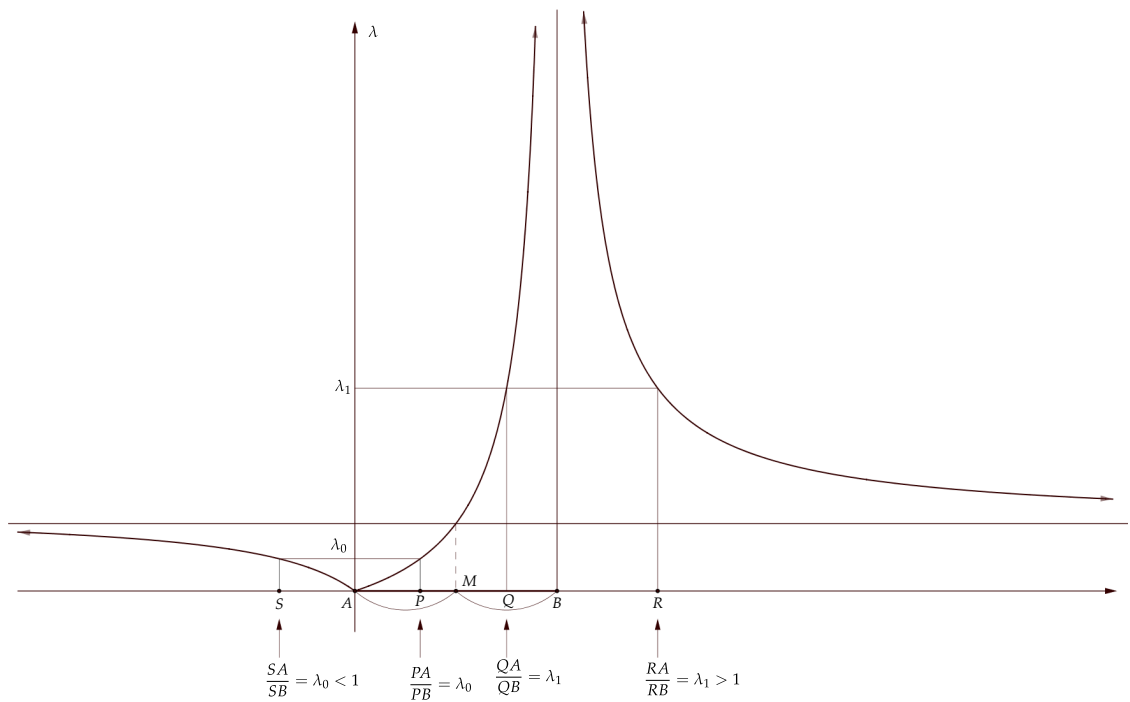}\\
\caption{$S$ y $P$ conjugados armónicos de $A, B$ para $\lambda_0>$; $Q$ y $R$ conjugados armónicos de $A,B$ para $\lambda_1>1$}\label{188}
\end{center}
\end{figure}
\newline
Así las cosas, si $\lambda=\dfrac{3}{2}>1$, por ejemplo, podemos construir dos puntos $Q$ y $R$, $Q$ entre $M$ y $B, R$ a la derecha de $B$ de modo que, $$\dfrac{QA}{QB}=\dfrac{RA}{RB}=\dfrac{3}{2}$$ Esta construcción se lleva a cabo con regla y compás así: se traza una semirecta $x$ con origen en $A$ y con una abertura de compás fija se marcan los puntos 1, 2, 3, 4, 5.
Luego se une 5 con $B$ y por el punto 3 se traza la paralela a $\overline{B5}$ hasta cortar a $\overline{AB}$ en $Q$, (véase Fig. \ref{189}).\\[.3cm] Se tiene entonces que, $$\dfrac{QA}{QB}=\dfrac{\overline{AB}}{\overline{35}}=\dfrac{3}{2}=\lambda$$
\begin{figure}[ht!]
\begin{center}
\includegraphics[scale=0.3]{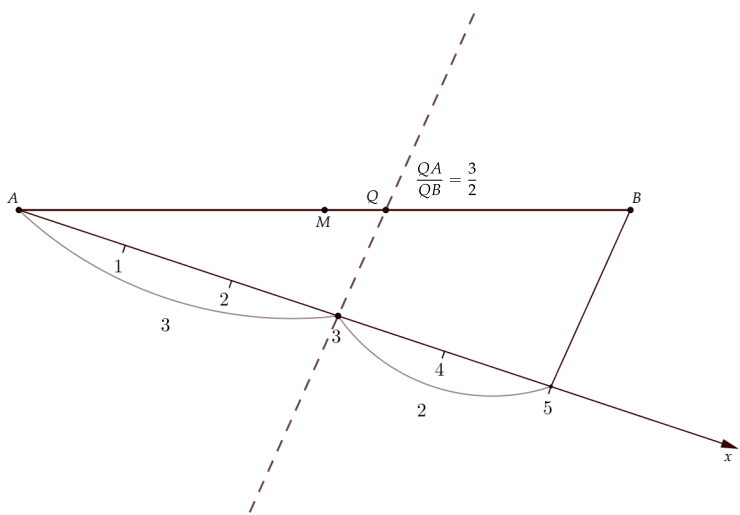}\\
\caption{Construcción con regla y compás del punto $Q$ entre $M$ y $B$}\label{189}
\end{center}
\end{figure}
\newline
Ahora se construye $R$, (véase Fig. \ref{190}).
\begin{figure}[ht!]
\begin{center}
\includegraphics[scale=0.3]{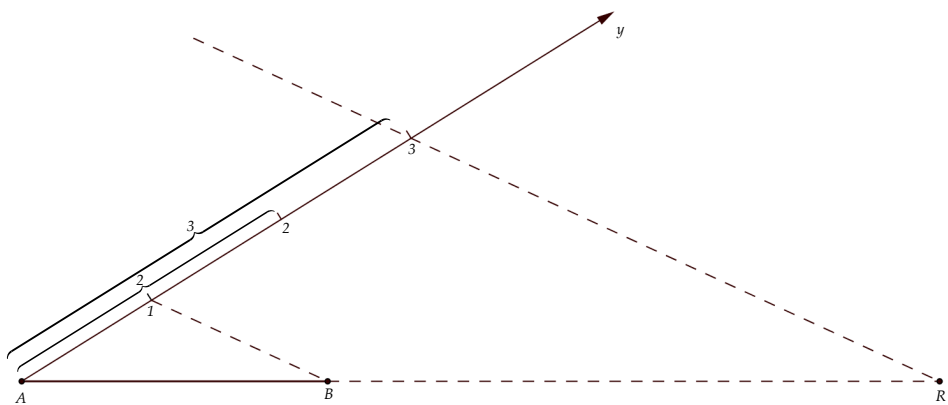}\\
\caption{Construcción con regla y compás del punto $R$, $R$ a la derecha de $B$, con $\lambda=\dfrac{3}{2}$}\label{190}
\end{center}
\end{figure}
\newline
Se traza por $A$ la semirecta $y$ y se señalan con el compás los puntos 1, 2 y 3. Se une 1 con $B$ y por 3 se traza $\overline{3R}$ paralelo a $\overline{1B}$.\\[.3cm] Se tiene entonces que, $$\dfrac{RA}{RB}=\dfrac{3}{2}=\lambda$$ La construcción de los puntos $P$ y $Q$ en los que se cumple que: $$\dfrac{PA}{PB}=\dfrac{QA}{QB}=\dfrac{2}{5}=\lambda$$ aparece indicada en la (Fig. \ref{191}).
\begin{figure}[ht!]
\begin{center}
\includegraphics[scale=0.3]{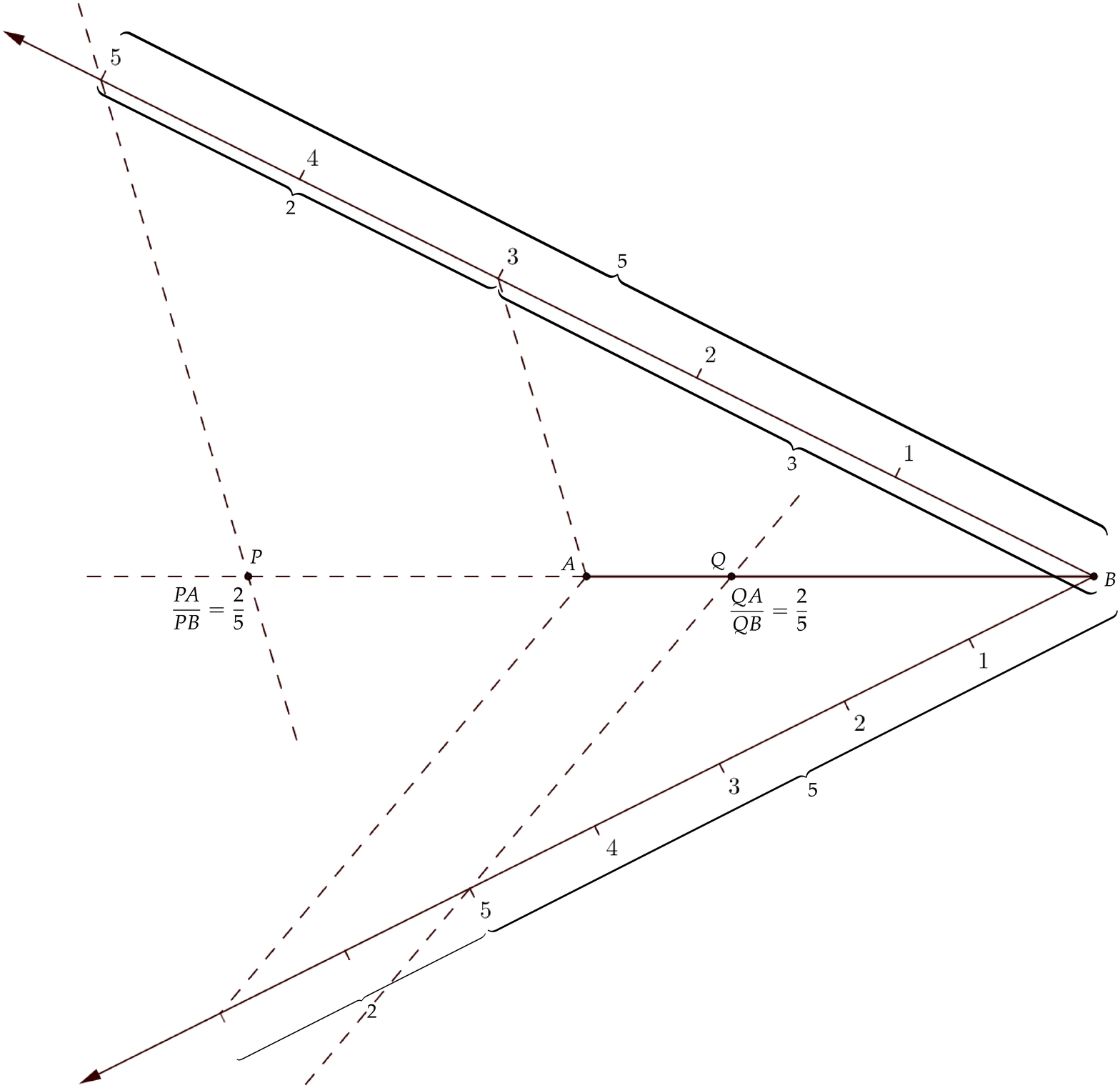}\\
\caption{Construcción con regla y compás de los puntos $Q$ y $P$, $P$ a la izquierda de $A$ cuando $\lambda=\dfrac{3}{5}$}\label{191}
\end{center}
\end{figure}
\newline
Otro hecho geométrico que necesitamos es el siguiente.
\begin{teor}
Si en un triángulo $\triangle ABC$, se trazan las bisectrices interior y exterior en $A$,  llamemos $D$ y $E$ a los pies de esas bisectrices. Entonces, $$\dfrac{DB}{DC}=\dfrac{EB}{EC}=\dfrac{c}{b}\quad\text{(véase Fig. \ref{192})}$$
\end{teor}
\noindent La demostración se omite, trate el lector de hacerla o consulte un libro de geometría.
\begin{figure}[ht!]
\begin{center}
\includegraphics[scale=0.3]{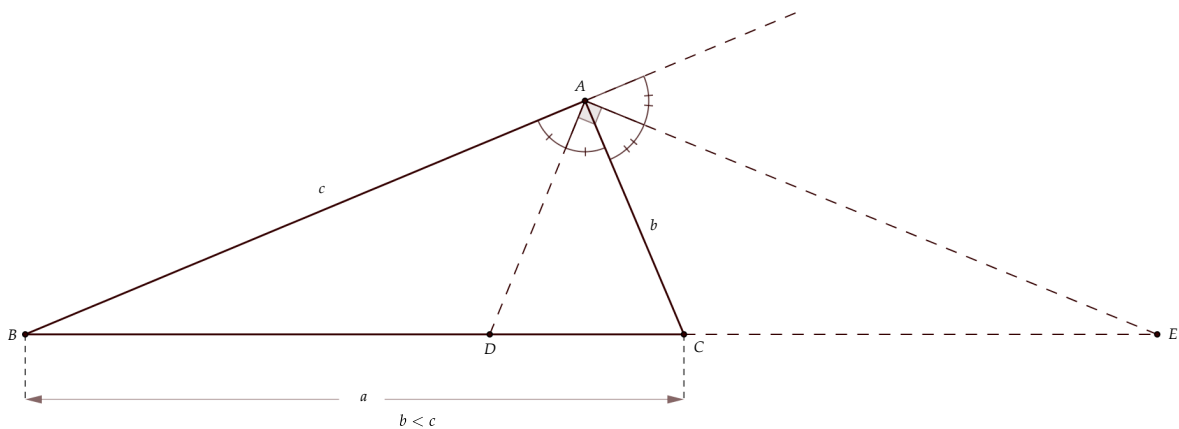}\\
\caption{$\overline{AD}, \overline{AE}$ bisectrices interior y exterior en $A$. $\dfrac{DB}{DC}=\dfrac{EB}{EC}=\dfrac{c}{b}$}\label{192}
\end{center}
\vspace{0.5cm}
\end{figure}
El teorema anterior tiene un recíproco.
\begin{teor}
Se tiene un triángulo $\triangle ABC$ en el que $b<c$ y dos puntos $F$ entre $B$ y $C$ y $G$ a la derecha de $C$.\\
\item[(1)] Si $\dfrac{FB}{FC}=\dfrac{c}{b}$, entonces $F$ es el pie de la bisectriz interior trazada desde $A$.\\
\item[(2)] Si $\dfrac{GB}{GC}=\dfrac{c}{b}$, entonces $G$ es el pie de la bisectriz exterior trazada desde $A$, (véase Fig. \ref{193}).
\end{teor}
\begin{figure}[ht!]
\begin{center}
\includegraphics[scale=0.3]{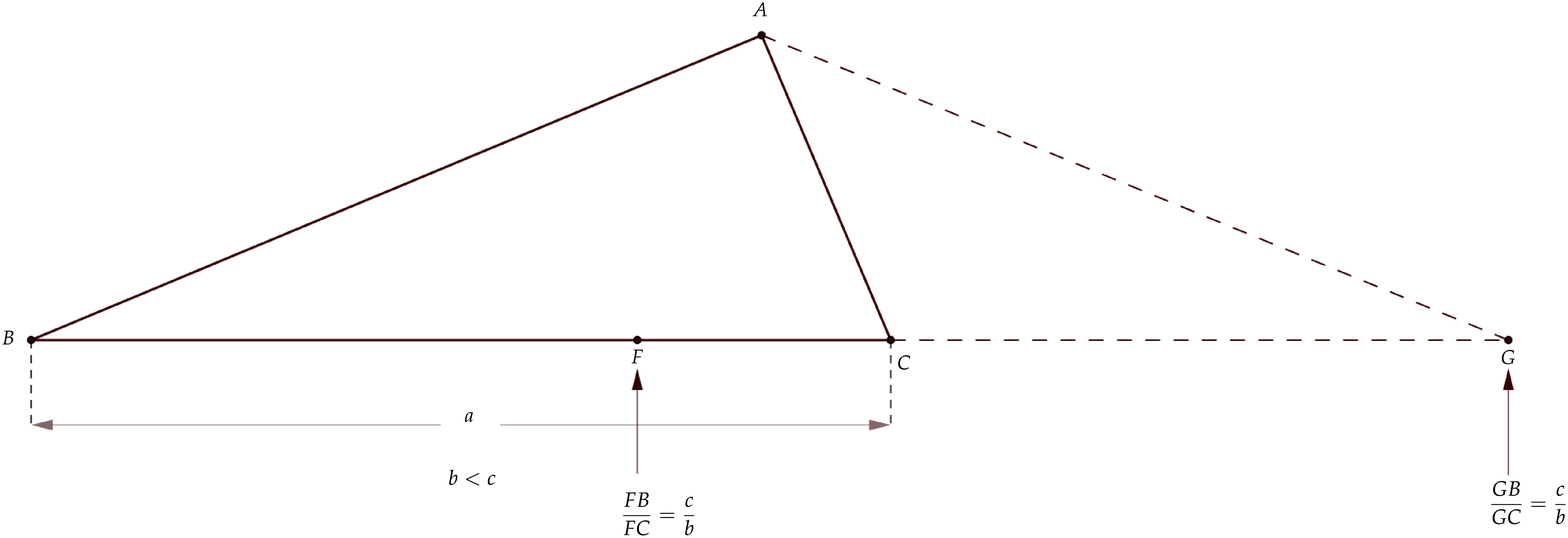}\\
\caption{}\label{193}
\end{center}
\vspace{0.5cm}
\end{figure}
\noindent La demostración se omite. Llévela a cabo.\\[.3cm] Y ahora si la \textit{circunferencia de Apollonius}.
\begin{prop}[\textbf{Círculo de Apollonius}]
Sea $\overline{AB}$ un segmento dado y $\dfrac{p}{q}$ un decimal dado con $p>q$. Llamemos $P$ y $Q$ a los conjugados armónicos de $A$ y $B$ para $\dfrac{p}{q}$. O sea, $P$ está entre $A$ y $B, Q$ a la derecha de $B$ y se tiene que, $$\dfrac{PA}{PB}=\dfrac{QA}{QB}=\dfrac{p}{q}$$ El lugar de los puntos $X$ del plano en los que se cumple que, $$\dfrac{XA}{XB}=\dfrac{p}{q}$$ es la circunferencia de diámetro $\overline{PQ}$.
\end{prop}
\begin{proof}
Se debe demostrar que $\left\{X\diagup\dfrac{XA}{XB}=\dfrac{p}{q}\right\}=\mathscr{C}:\text{la circunferencia de diámetro $PQ$}$.\\
\item[(1)] Sea $X$ un punto del plano en el que $\dfrac{XA}{XB}=\dfrac{p}{q}$.\\ por hipótesis, $$\dfrac{PA}{PB}=\dfrac{QA}{QB}=\dfrac{p}{q}$$ Veamos que $X\in\mathscr{C}$ donde $\mathscr{C}$ en la circunferencia de diámetro $\overline{PQ}$, (véase Fig. \ref{196}).\\[.3cm]
\begin{figure}[ht!]
\begin{center}
\includegraphics[scale=0.3]{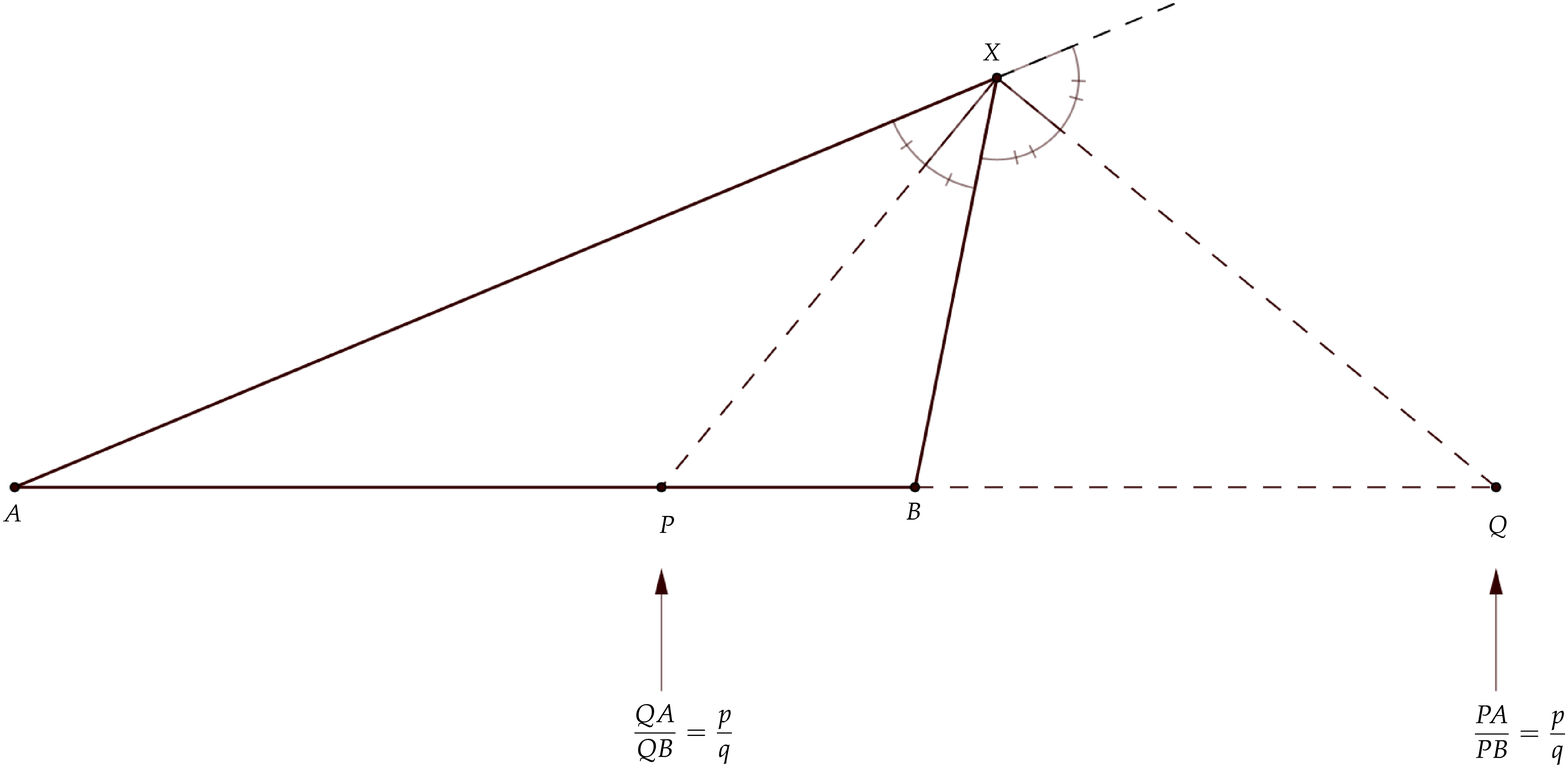}\\
\caption{}\label{196}
\end{center}
\end{figure}
\newline
Unimos $X$ con $A,B,P$ y $Q$. Bastará con que $\widehat{PXQ}=90^\circ$.\\[.3cm] Consideremos el triángulo $\triangle AXB$.\\[.3cm] Como $\dfrac{PA}{PB}=\dfrac{XA}{XB},\quad P$ es el pie de la bisectriz interior en $A$.\\ Como $\dfrac{QA}{QB}=\dfrac{XA}{XB},\quad Q$ es el pie de la bisectriz exterior en $A$.\\[.3cm] Por lo tanto, $\overline{PX}\perp QX$ y $\widehat{PXQ}=90^\circ$ lo que nos prueba que $X$ está en la circunferencia de diámetro $\overline{PQ}$.\\
\item[(2)] Tomemos ahora $X\in\mathscr{C}$: la circunferencia de diámetro $\overline{PQ}$.\\[.3cm] Por hipótesis, $$\dfrac{PA}{PB}=\dfrac{QA}{QB}=\dfrac{p}{q}$$ Veamos que $\dfrac{XA}{XB}=\dfrac{p}{q}$, (véase Fig. \ref{197}).\\[.3cm]
\begin{figure}[ht!]
\begin{center}
\includegraphics[scale=0.3]{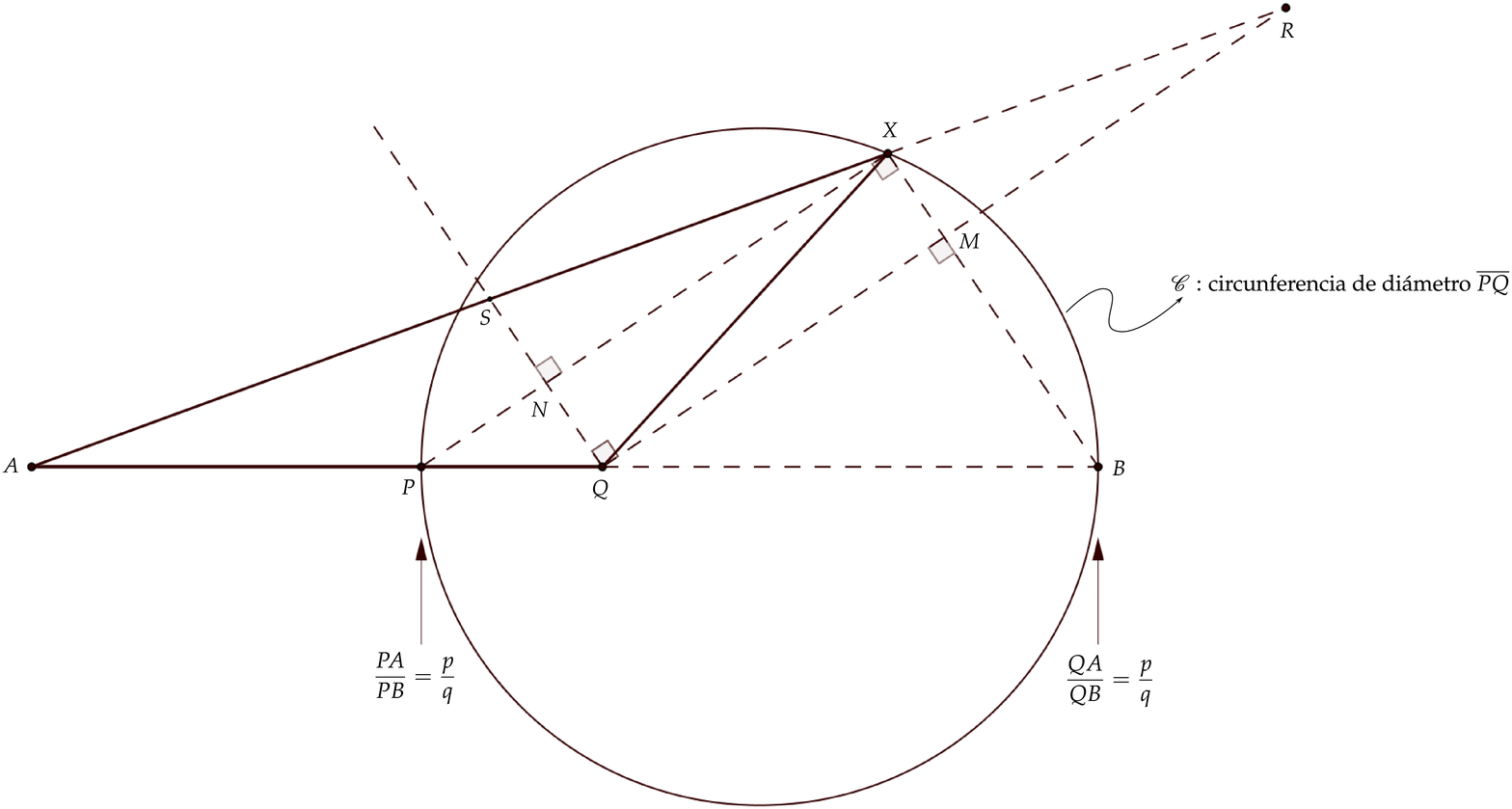}\\
\caption{}\label{197}
\end{center}
\end{figure}
\newline
Se une $X$ con $P$ y $Q$. Como $X\in\mathscr{C}$ y $\widehat{PXQ}$ está inscrito en el diámetro $\overline{PQ}, \overline{QX}\perp\overline{PX}$.\\
\item[$\bullet$] Se traza por $B, BMR\parallel\overline{PX}$.\\[.3cm] Como $\overline{BM}\parallel\overline{PX}$ y $\overline{QX}\perp\overline{PX},
\overline{QX}\perp\overline{BM}$ y por lo tanto, $\widehat{M}=90^\circ$.\\[.3cm] Ahora, por el Teorema Básico de Semejanza (TBS),
\begin{equation}\label{98}
\dfrac{XA}{XR}=\dfrac{PA}{PB}=\dfrac{p}{q}.\quad\text{Así que $\dfrac{XA}{XR}=\dfrac{p}{q}$}
\end{equation}
\item[$\bullet$] A continuación se traza por $B, BNS\parallel\overline{QX}$.\\[.3cm] Como $\widehat{BN}\parallel\overline{QX}$ y $\overline{PX}\parallel\overline{QX}, \overline{PX}\perp\overline{QN}$ y por lo tanto, $\widehat{N}=90^\circ$.\\[.3cm] Como $\widehat{N}=\widehat{X}=\widehat{M}=90^\circ$, entonces, $\widehat{NBM}=90^\circ$ y el cuadrilátero $NXMB$ es un rectángulo y el triángulo $\triangle SBR$ es rectángulo en $B$. Ahora, por el TBS,
\begin{align*}
\dfrac{XA}{XS}=\dfrac{QA}{QB}&\underset{\uparrow}{=}\dfrac{p}{q}\\
&\text{hipótesis}
\end{align*}
Así que,
\begin{equation}\label{99}
\dfrac{XA}{XS}=\dfrac{p}{q}
\end{equation}
De \eqref{98} y \eqref{99}, $\dfrac{XA}{XR}=\dfrac{XA}{XS}$. O sea que $XR=XS$ lo que nos demuestra que en el triángulo rectángulo $SBR, \overline{BX}$ es mediana relativa a la hipotenusa, y por lo tanto, $XB=XR=XS$ y regresando a \eqref{98}: $\dfrac{XA}{XB}=\dfrac{p}{q}$
\end{proof}
\end{enumerate}
\begin{ejer}
Cálculo del radio del círculo de Apollonius.\\[.3cm] Ecuación analítica de la circunferencia de Apollonius.
\end{ejer}
\begin{solu}
Sea $\mathscr{C}$ la circunferencia de Apollonius de $\overline{AB}$ para $\dfrac{m}{n},\quad m>n$, (véase Fig. \ref{198}).\\[.3cm]
\begin{figure}[ht!]
\begin{center}
\includegraphics[scale=0.3]{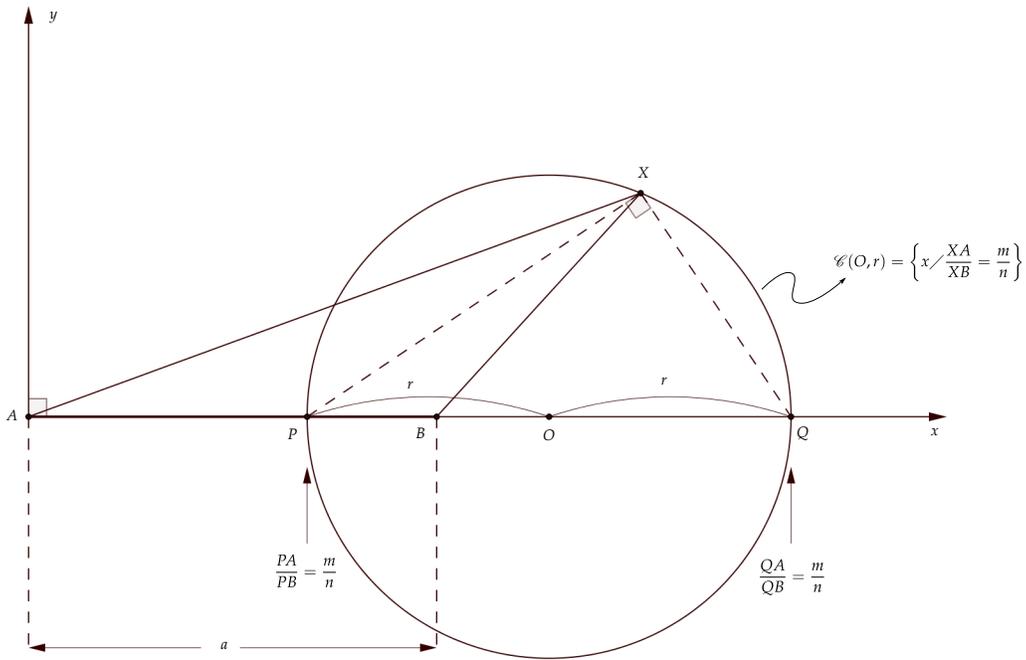}\\
\caption{Circunferencia de Apollonius para $A, B$ y razón $\lambda=\dfrac{m}{n}$}\label{198}
\end{center}
\end{figure}
\newline
Supongamos que hemos construido a $P$ y $Q$; los conjugados armónicos de $A,B$ para $\dfrac{m}{n}$.\\[.3cm] Llamemos $O$ el centro de la circunferencia de Apollonius y $r$ a su radio.\\[.3cm] Entonces,
\begin{align*}
\dfrac{PA}{PB}&=\dfrac{QA}{QB}=\dfrac{m}{n}\\
\shortintertext{Luego,}
\dfrac{PA+PB}{PB}&=\dfrac{m+n}{n},\quad\dfrac{AB}{PB}=\dfrac{m+n}{n}\\
\therefore PB&=\dfrac{n}{m+n}AB\\
\shortintertext{También,}
\dfrac{QA-QB}{QB}&=\dfrac{m-n}{n},\quad\dfrac{AB}{QB}=\dfrac{m-n}{n}\\
\therefore QB&=\dfrac{n}{m-n}AB\\
\shortintertext{Así que,}
2r&=PB+QB\\
&=\left(\dfrac{n}{m+n}+\dfrac{n}{m-n}\right)AB\\
&=\dfrac{2mn}{m^2-n^2}AB\\
\shortintertext{y por lo tanto,}
r&=\dfrac{mn}{m^2-n^2}a
\end{align*}
\noindent Ahora localicemos el centro $O$ de $\mathscr{C}$.
\begin{align*}
OB=r\cdot\underset{\stackrel{\parallel}{PB}}{\underbrace{\dfrac{n}{m+n}}}AB&=\left(\dfrac{mn}{m^2-n^2}-\dfrac{n}{m+n}\right)AB\\
&=\dfrac{n^2}{m^2-n^2}AB\\
AO=AB+OB&=\left(1+\dfrac{n^2}{m^2-n^2}\right)AB=\dfrac{m^2}{m^2-n^2}a
\end{align*}
Finalmente, la ecuación de la circunferencia $\mathscr{C}$ con respecto al sistema $xy$ con origen en $A$ es: $${\left(x-\dfrac{m^2}{m^2-n^2}a\right)}^2+y^2={\left(\dfrac{mn}{m^2-n^2}a\right)}^2$$
\end{solu}
\subsection{Segundo lugar geométrico: circunferencia de centro en el punto medio de $\overline{AB}$}
Para tratar el segundo lugar necesitamos el <<teorema de la mediana>>.
\begin{teor}
La suma de los cuadrados de dos lados de un triángulo es igual a dos veces el cuadrado de la mediana relativa al tercer lado más la mitad del cuadrado del tercer lado, (veáse Fig. \ref{199}).
\begin{figure}[ht!]
\begin{center}
\includegraphics[scale=0.3]{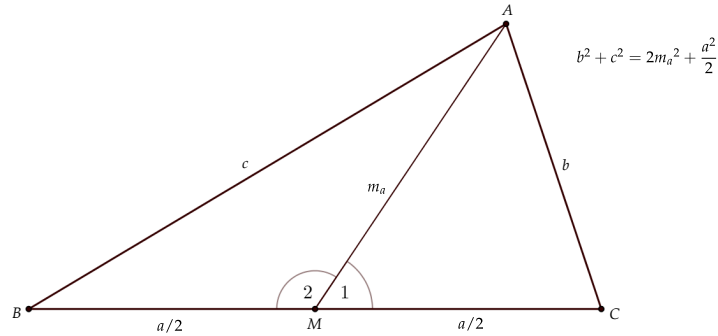}\\
\caption{Teorema de la mediana}\label{199}
\end{center}
\end{figure}
\end{teor}
\begin{proof}
Para obtener este resultado vamos a aplicar dos veces la ley de Cosenos.\\[.3cm] En el triángulo $\triangle AMC$,
\begin{equation}\label{100}
b^2=\dfrac{a^2}{4}+m_a^2-2\cdot\dfrac{a}{2}\cdot m_a\cdot\cos\widehat{M_1}
\end{equation}
En el triángulo $\triangle AMB$,
\begin{equation}\label{101}
c^2=\dfrac{a^2}{4}+m_a^2-2\cdot\dfrac{a}{2}\cdot m_a\cdot\cos\widehat{M_2}
\end{equation}
Sumando y teniendo en cuenta que $\cos\widehat{M_2}=-\cos\widehat{M_1}$,
\begin{align*}
b^2+c^2=\dfrac{2a^2}{4}+2m_a^2\\
\shortintertext{O sea:}
b^2+c^2=2m_a^2+\dfrac{a^2}{2}
\end{align*}
\end{proof}
\noindent Y ahora sí, estudiemos el segundo lugar geométrico.
\begin{prop}[\textbf{Circunferencia de centro en el punto medio de $\overline{AB}$}]
Sea $\overline{AB}$ un segmento dado y $O$ su punto medio. El lugar de los puntos $M$ para los cuales ${MA}^2+{MB}^2=K^2$ donde $K^2$ es una constante dada es una circunferencia de centro en $O$, ($O$ punto medio) de $\overline{AB}$ (véase Fig. \ref{200}).
\end{prop}
\begin{figure}[ht!]
\begin{center}
\includegraphics[scale=0.3]{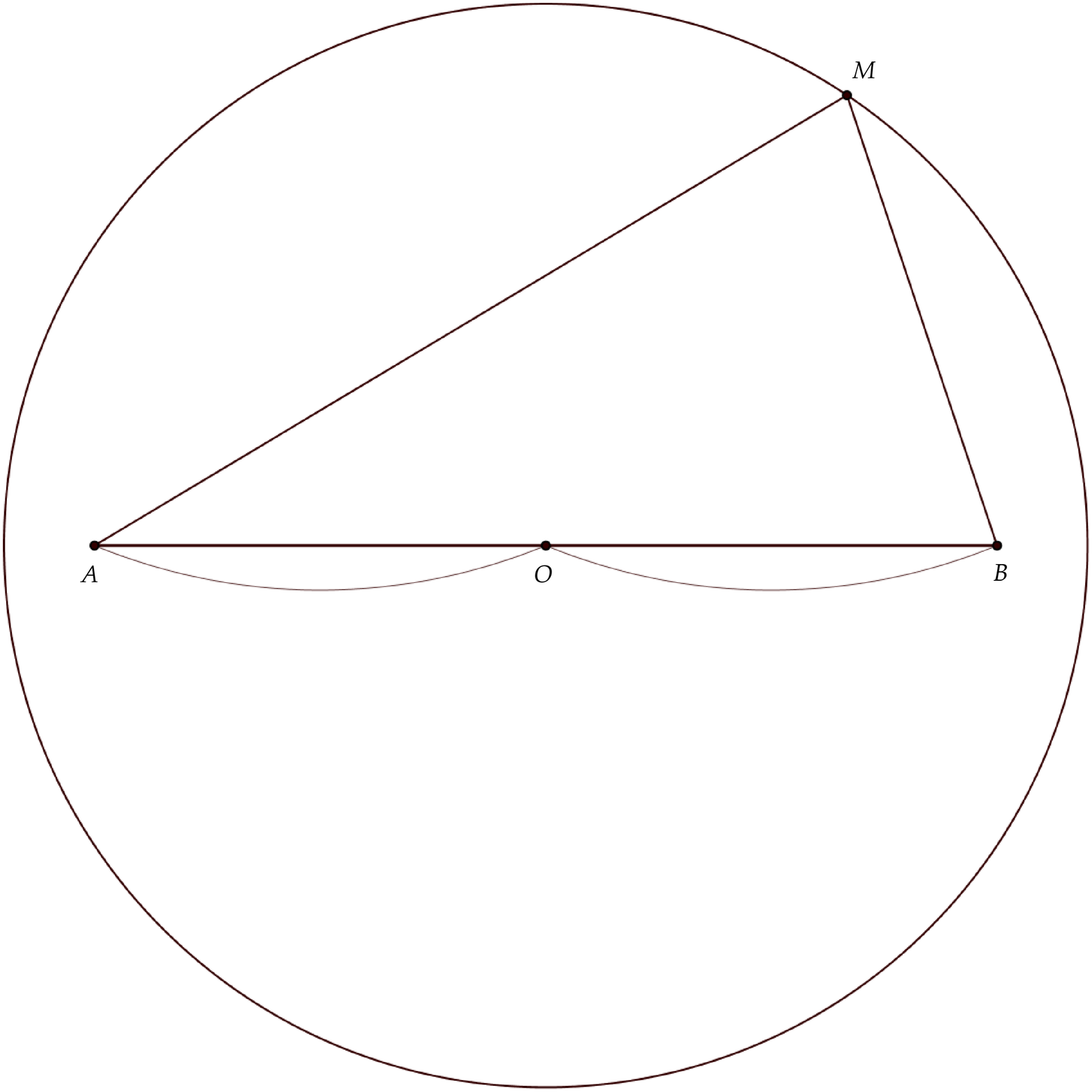}\\
\caption{}\label{200}
\end{center}
\end{figure}
\begin{proof}
\noindent
\begin{enumerate}
\item[$\bullet$] Sea $M$ un punto del lugar. Unimos $M$ con $A$ y $B$.\\[.3cm] Entonces,
\begin{equation}\label{176}
{MA}^2+{MB}^2=k^2
\end{equation}
Unimos $M$ con $O$: punto medio de $\overline{AB}$.\\[.3cm] Por el Teorema de la mediana en el triángulo $\triangle MAB$,
\begin{equation}\label{177}
{MA}^2+{MB}^2=2\cdot{OM}^2+\dfrac{{AB}^2}{2}
\end{equation}
De \eqref{176} y \eqref{177}:
\begin{align*}
2\cdot{OM}^2+\dfrac{{AB}^2}{2}&=k^2\\
\therefore\quad{OM}^2&=\dfrac{k^2}{2}-\dfrac{{AB}^2}{4}:\text{constante}
\end{align*}
\item[$\bullet$] Tomemos ahora un punto $M\in\mathscr{C}\left(0;\sqrt{\dfrac{k^2}{2}-\dfrac{{AB}^2}{4}}\right)$ y veamos que $M$ cumple la propiedad, o sea que ${MA}^2+{MB}^2=k^2$. Unimos $M$ con $O$.\\[.3cm] Como:
\begin{align}
&M\in\mathscr{C}\left(0;\sqrt{\dfrac{k^2}{2}-\dfrac{{AB}^2}{4}}\right),\notag\\
&OM=\sqrt{\dfrac{k^2}{2}-\dfrac{{AB}^2}{4}}\notag\\
\shortintertext{O sea que,}
&{OM}^2=\dfrac{k^2}{2}-\dfrac{{AB}^2}{4}\label{178}
\end{align}
\noindent Por el Teorema de la mediana,
\begin{align*}
{MA}^2+{MB}^2=2\cdot{OM}^2+\dfrac{{AB}^2}{2}&\underset{\downarrow}{=}2\left(\dfrac{k^2}{2}-\dfrac{{AB}^2}{4}\right)+\dfrac{{AB}^2}{2}=k^2,\\
&\eqref{178}
\end{align*}
\end{enumerate}
\end{proof}
\begin{discu}\noindent
\begin{enumerate}
\item[(1)] Habrá lugar si $\dfrac{k^2}{2}-\dfrac{{AB}^2}{4}\geqslant 0$, o sea si $$\dfrac{k^2}{2}\geqslant\dfrac{{AB}^2}{4},\quad\text{i.e.,}\quad\text{si}\quad k\geqslant\dfrac{AB}{\sqrt{2}}=\dfrac{AB\sqrt{2}}{2}\approx 0.707 AB$$
\item[(2)] Si $k=\dfrac{\sqrt{2}}{2}AB$, el lugar se reduce al punto $O$.\\
\item[(3)] En el caso que ${MA}^2+{MB}^2={AB}^2$, o sea, si $k=AB$, el lugar es la circunferencia de diámetro $AB$.
\end{enumerate}
\end{discu}
\noindent Resumiendo, si $k>\dfrac{AB\cdot\sqrt{2}}{2}$, se tiene $\left\{M\diagup{MA}^2+{MB}^2=K^2\right\}=\mathscr{C}(O,R)$ con $R=\left(O;\sqrt{\dfrac{K^2}{2}-\dfrac{{AB}^2}{4}}\right)$ (véase Fig. \ref{201}).
\begin{figure}[ht!]
\begin{center}
\includegraphics[scale=0.3]{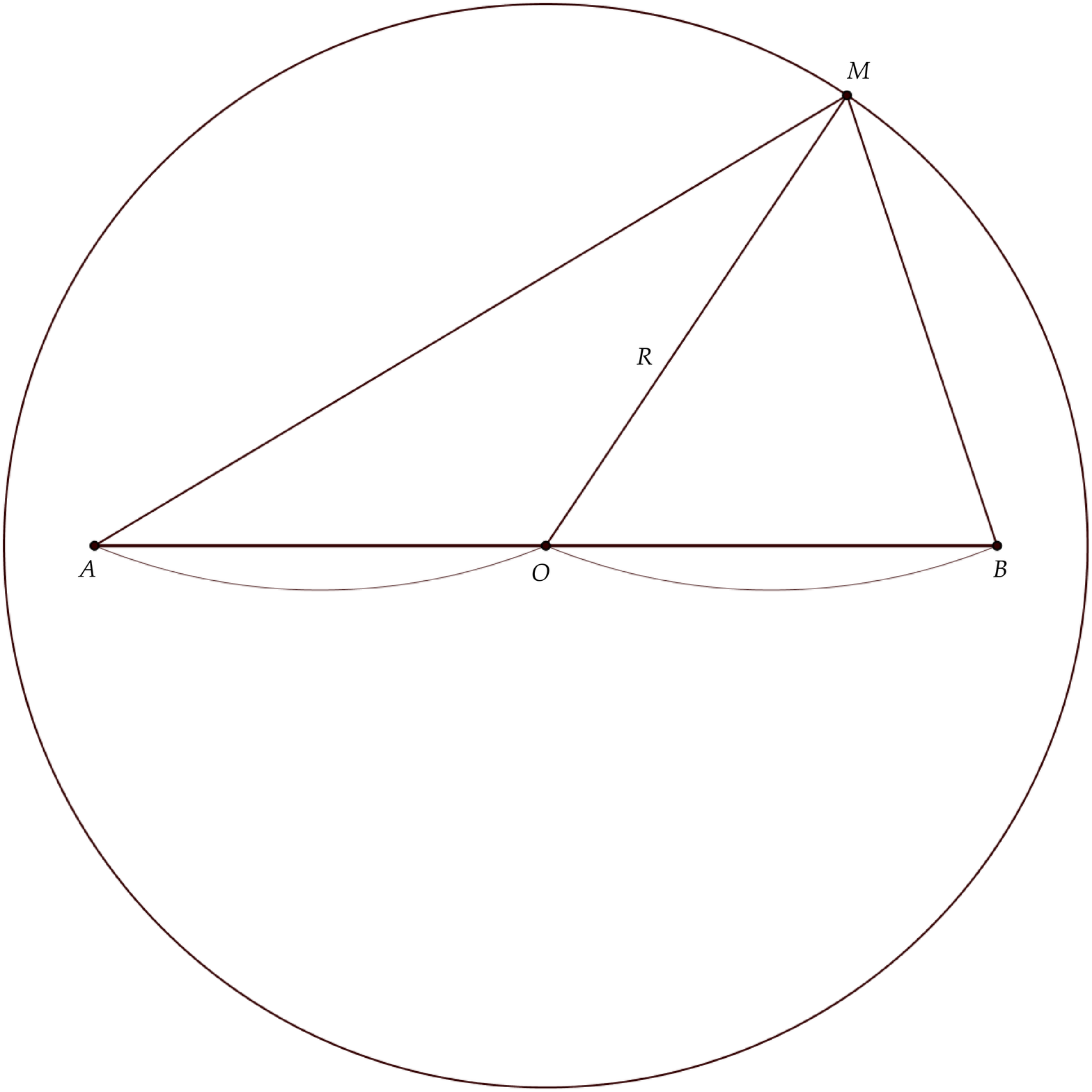}\\
\caption{Circunferencia $\mathscr{C}\left(O,R\right)=\left(O;\sqrt{\dfrac{K^2}{2}-\dfrac{{AB}^2}{4}}\right)=\left\{M\diagup{MA}^2+{MB}^2=K^2\right\}$}\label{201}
\end{center}
\end{figure}
\subsection{Tercer lugar geométrico: recta perpendicular a $\overline{AB}$}
Para estudiar el tercer lugar geométrico necesitamos el siguiente resultado de la geometría:
\begin{teor}
Se tiene un triángulo $\triangle ABC$ con $c>b$. La diferencia de los cuadrados de dos lados de un triángulo es dos veces el producto del tercer lado por la proyección de la mediana relativa al tercer lado sobre este, (véase Fig. \ref{202}).
\begin{figure}[ht!]
\begin{center}
\includegraphics[scale=0.3]{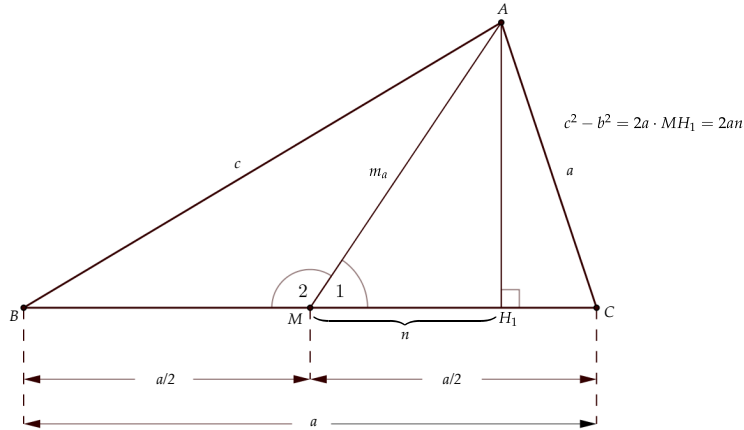}\\
\caption{${c}^2-{b}^2=2a\cdot MH_1=2an$}\label{202}
\end{center}
\end{figure}
\end{teor}
\noindent Demostremos esta afirmación.
\begin{proof}
De nuevo se aplica dos veces la ley de Cosenos.
\begin{align}
\text{En el triángulo $\triangle AMC$,}\quad b^2&=\dfrac{a^2}{4}+{m_a}^2-a\cdot m_a\cos M_1\label{179}\\
\text{En el triángulo $\triangle AMB$,}\quad c^2&=\dfrac{a^2}{4}+{m_a}^2-a\cdot m_a\cos M_2\notag\\
&\underset{\uparrow}{=}\dfrac{a^2}{4}+{m_a}^2+a\cdot m_a\cdot\cos M_1\label{180}\\
&\cos\widehat{M_1}=-\cos\widehat{M_1}\notag\\
\eqref{180}-\eqref{179}:c^2-b^2&=2a\left(m_a\cdot\cos\widehat{M_1}\right)\notag\\
\shortintertext{O sea}
c^2-b^2&=2a\cdot n\notag
\end{align}
\end{proof}
\noindent Y con este resultado consideramos el tercer lugar geométrico.
\begin{prop}[\textbf{recta perpendicular a $\overline{AB}$}]
El lugar de los puntos $M$ desde los cuales la diferencia de los cuadrados de su distancia a dos puntos fijos $A$ y $B$ es constante e igual a $k^2(k^2>0)$ es una recta perpendicular a la recta $\overset{\longleftrightarrow}{AB}$. Si $O$ es el punto medio de $\overline{AB}$ y $D$ es el pie de la perpendicular de dicha recta perpendicular, $OD=\dfrac{K^2}{2\cdot AB}$, (véase Fig. \ref{203}).
\end{prop}
O sea:
\begin{figure}[ht!]
\begin{center}
\includegraphics[scale=0.3]{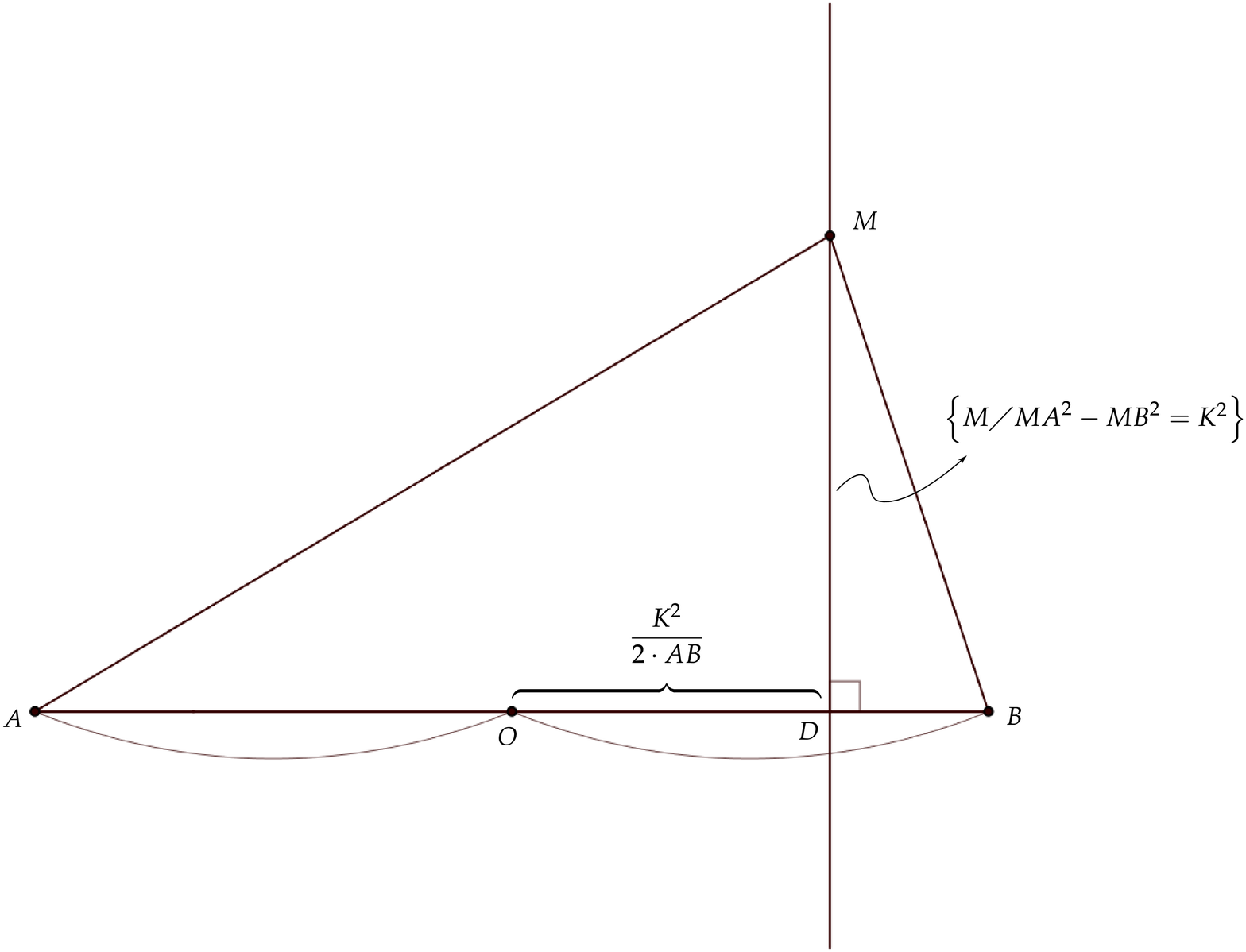}\\
\caption{}\label{203}
\end{center}
\end{figure}
\begin{proof}
\noindent
\begin{enumerate}
\item[$\bullet$] Sea $M$ un punto del lugar.\\[.3cm] Entonces ${MA}^2-{MB}^2=k^2$. Unimos $M$ con $O$ y bajamos desde $M, \overline{MD}\perp\overline{AB}$.\\[.3cm] $\overline{OD}$: proyección de la mediana $\overline{OM}$ del triángulo $\triangle ABC$ sobre $\overline{AB}$.\\[.3cm] Entonces,
\begin{align*}
{MA}^2-{MB}^2&=2\cdot AB\cdot OD\\
\shortintertext{O sea:}
2AB\cdot OD&=k^2\quad\therefore\quad OD=\dfrac{k^2}{2\cdot AB},
\end{align*}
lo que indica que $M$ está sobre la perpendicular por $D$ a la recta $\overset{\longleftrightarrow}{AB}$.\\
\item[$\bullet$] Supongamos ahora que tomamos un punto $M$ en la recta $\ell\perp$ a $\overset{\longleftrightarrow}{AB}$ por $D$ donde $OD=\dfrac{k^2}{2\cdot AB}$, (véase Fig. \ref{204}).\\[.3cm]
\begin{figure}[ht!]
\begin{center}
\includegraphics[scale=0.3]{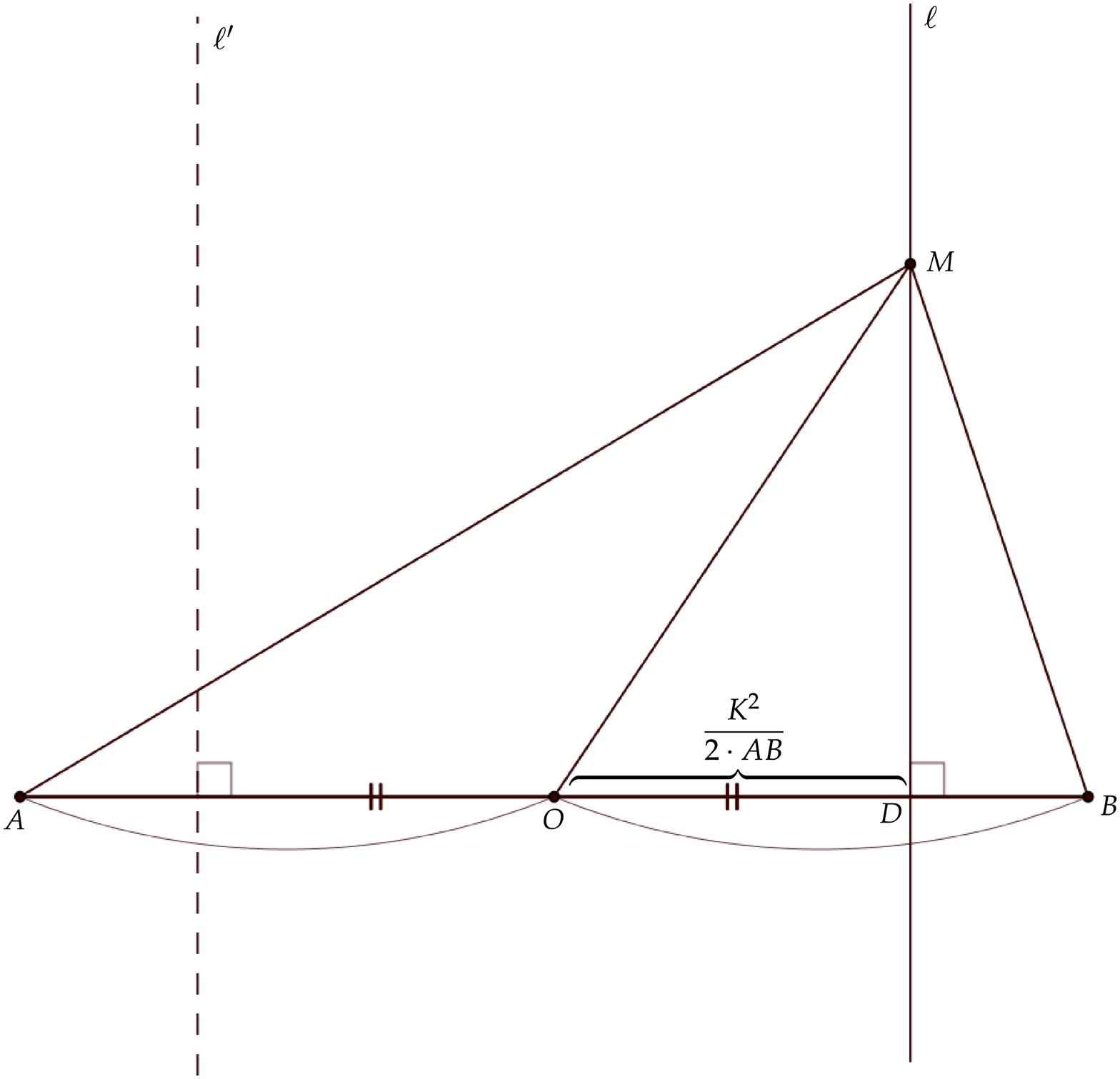}\\
\caption{}\label{204}
\end{center}
\end{figure}
\newline
Veamos que $M$ satisface la propiedad, o sea, veamos que ${MA}^2-{MB}^2=k^2$.\\[.3cm] Unimos $M$ con $O, A$ y $B$.\\ En el triángulo $\triangle MAB, \overline{MO}$ es mediana y $\overline{OD}$ es la proyección de $\overline{MO}$ sobre $\overset{\longleftrightarrow}{AB}$.\\[.3cm] Luego,
\begin{align*}
{MA}^2-{MB}^2&=2\cdot AB\cdot OD\\
&\underset{\uparrow}{=}2\cdot AB\cdot\dfrac{k^2}{2\cdot AB}=k^2,\\
&OD=\dfrac{k^2}{2\cdot AB}
\end{align*}
\end{enumerate}
\end{proof}
\begin{obser}
\begin{enumerate}\noindent
\item[(i)] El lugar de los puntos $M$ tal que ${MB}^2-{MA}^2=k^2$ es la recta $\ell'$ simétrica de la recta $\ell$ con respecto a $O$.\\
\item[(ii)] Si $k^2={AB}^2$, o sea, si $k=AB$, el lugar es la recta perpendicular por $B\cdot\overline{AB}$, (véase Fig. \ref{205}).
\begin{figure}[ht!]
\begin{center}
\includegraphics[scale=0.3]{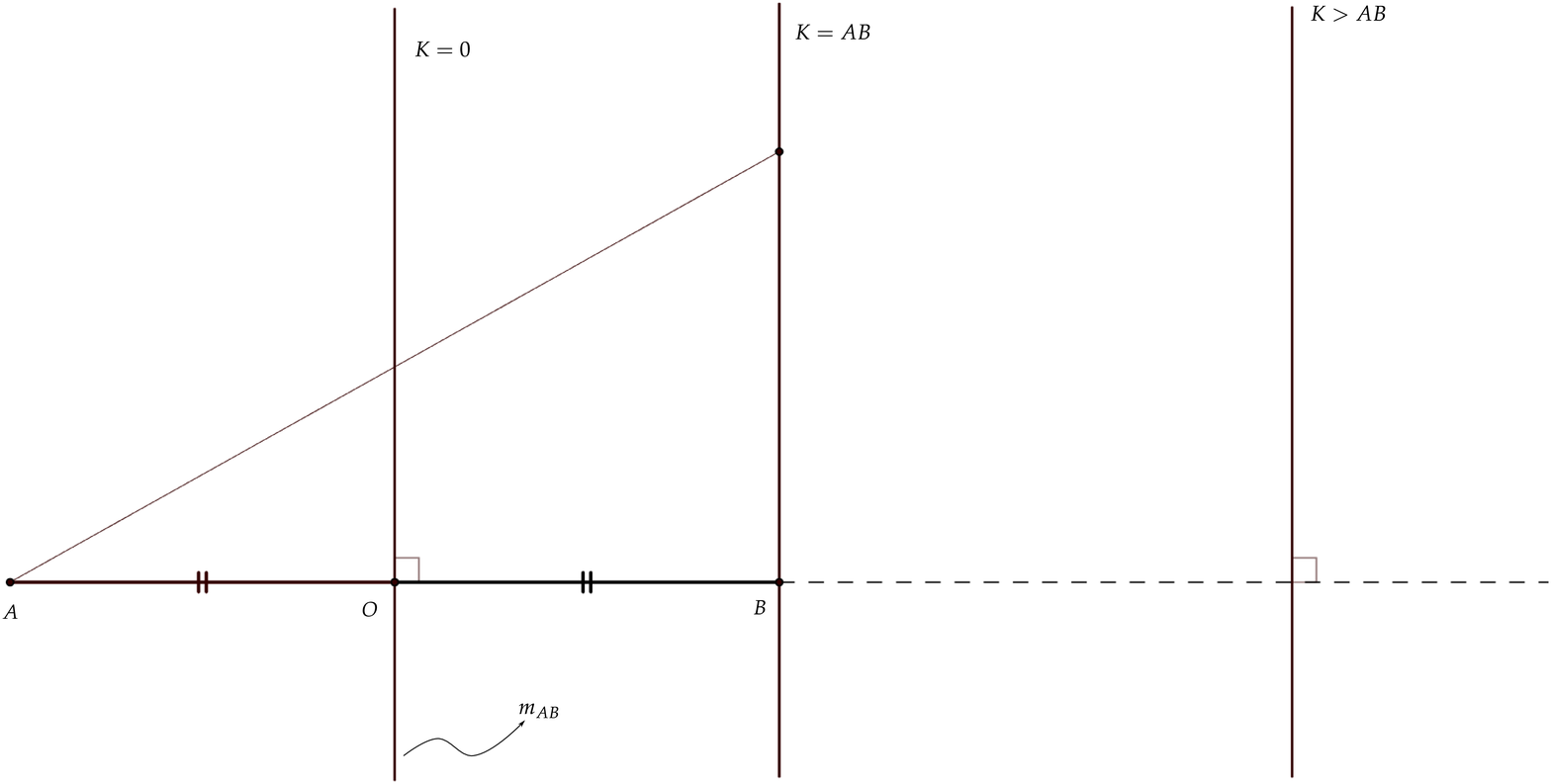}\\
\caption{}\label{205}
\end{center}
\end{figure}
\end{enumerate}
\end{obser}
\noindent Muy relacionado con lo que se viene discutiendo es el siguiente
\begin{pro}\label{208}
Se tiene un triángulo $\triangle ABC$. Llamemos $G$ al baricentro del triángulo (punto donde concurren las tres medianas del triángulo).\\[.3cm] Demostrar que:
\begin{enumerate}
\item[(1)] ${GA}^2+{BG}^2+{GC}^2=\dfrac{1}{3}(a^2+b^2+c^2)$\\
\item[(2)] $\forall X$ del plano:
\begin{align*}
{XA}^2+{XB}^2+{XC}^2&=\left({GA}^2+{GB}^2+{GC}^2\right)+3{XG}^2\\
&=\dfrac{1}{3}(a^2+b^2+c^2)+3{XG}^2\quad\text{(veáse Fig. \ref{206})}
\end{align*}
\end{enumerate}
\begin{figure}[ht!]
\begin{center}
\includegraphics[scale=0.3]{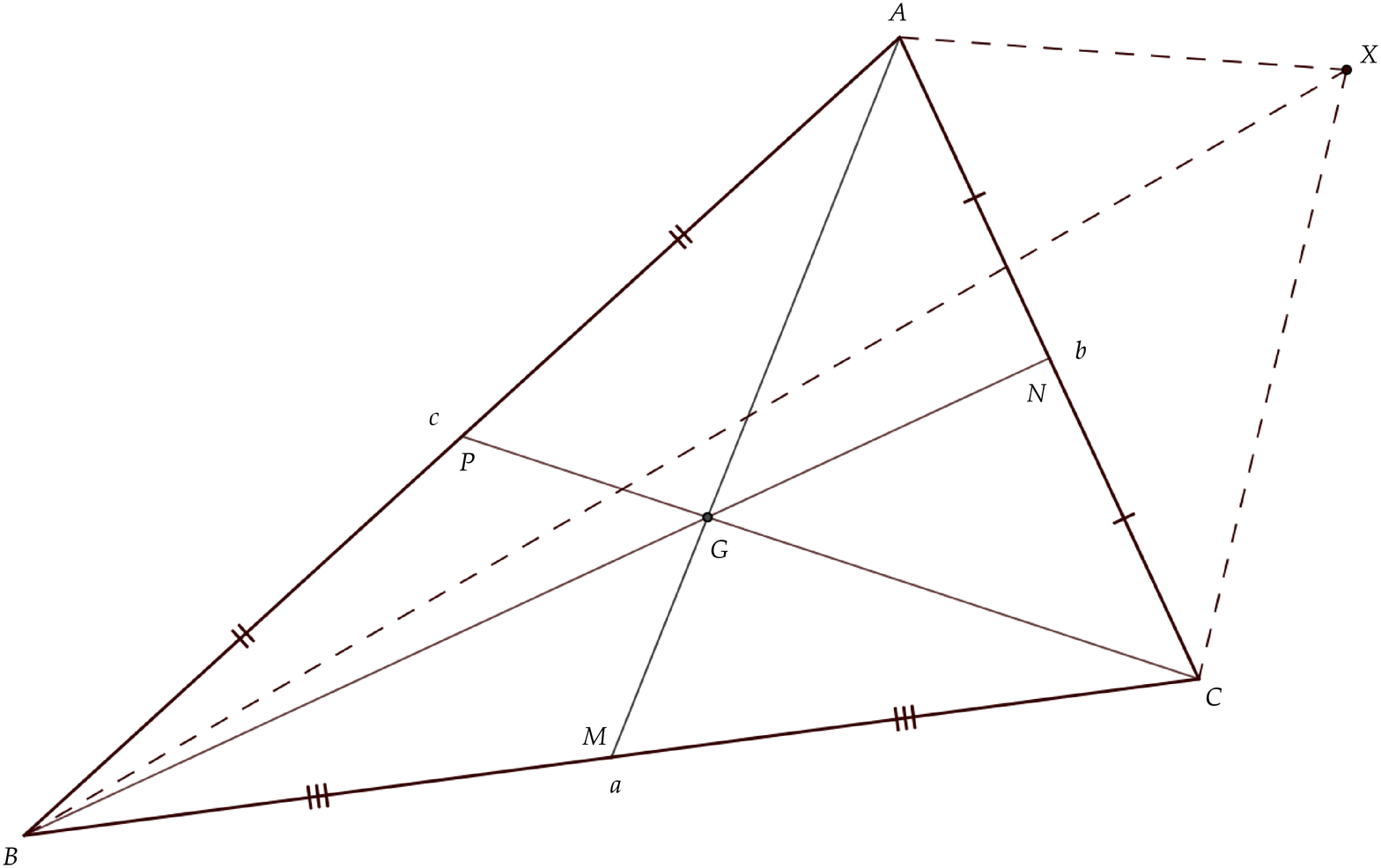}\\
\caption{}\label{206}
\end{center}
\end{figure}
\noindent En particular se tiene que si $X$ se toma en $G, XG=0$ y
\begin{align*}
\min\left\{{XA}^2+{XB}^2+{XC}^2\right\}&={GA}^2+{GB}^2+{GC}^2\\
&=\dfrac{1}{3}(a^2+b^2+c^2)
\end{align*}
\noindent y se alcanza en $G$.\\[.3cm]
\end{pro}
Para establecer $(1)$ y $(2)$ primero demostremos que:
\begin{prop}\label{181}
La suma de los cuadrados de las medianas de un triángulo es $\dfrac{3}{4}$ de la suma de los cuadrados de los lados del triángulo, (véase Fig. \ref{207}).\\[.3cm]
\end{prop}
\begin{proof}
Se debe probar que: $${m_a}^2+{m_b}^2+{m_c}^2=\dfrac{3}{4}(a^2+b^2+c^2)$$ Por el Teorema de la mediana:
\begin{align*}
b^2+c^2&=2{m_a}^2+\dfrac{a^2}{2}\\
a^2+c^2&=2{m_b}^2+\dfrac{b^2}{2}\\
a^2+b^2&=2{m_2}^2+\dfrac{c^2}{2}\\
\shortintertext{Luego,}
2(a^2+b^2+c^2)&=2({m_a}^2+{m_b}^2+{m_c}^2)+\dfrac{a^2+b^2+c^2}{2}\\
3(a^2+b^2+c^2)&=4({m_a}^2+{m_b}^2+{m_c}^2)\\
\therefore\quad{m_a}^2+{m_b}^2+{m_c}^2&=\dfrac{3}{4}(a^2+b^2+c^2)
\end{align*}
\begin{figure}[ht!]
\begin{center}
\includegraphics[scale=0.3]{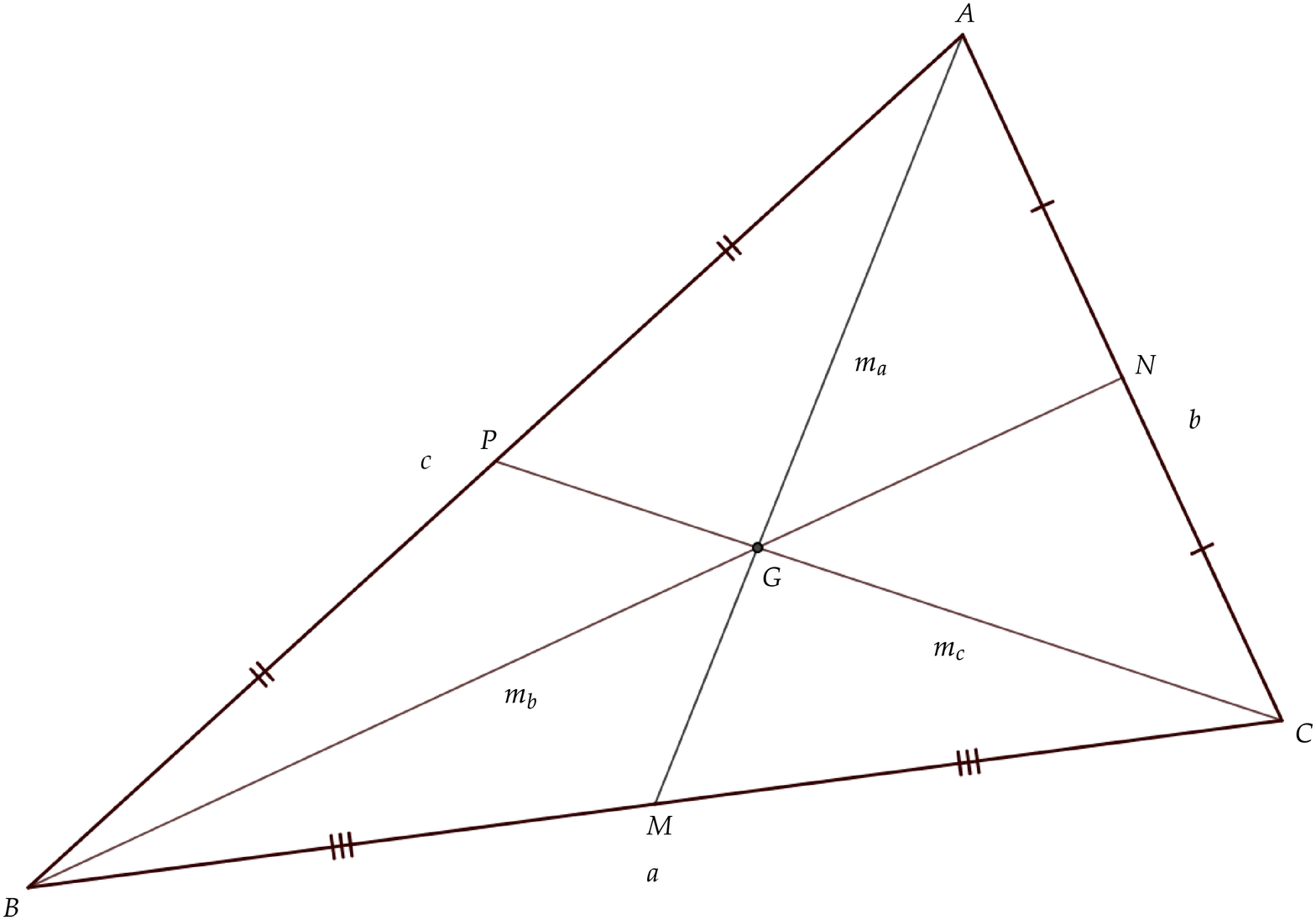}\\
\caption{}\label{207}
\end{center}
\vspace{0.5cm}
\end{figure}
\end{proof}
\begin{corol}
La suma de los cuadrados de las distancias al centroide (o baricentro) de un triángulo a los vértices es $\dfrac{1}{3}$ de la suma de los cuadrados de los lados.
\end{corol}
\begin{proof}
En efecto, como:
\begin{align*}
GA&=\dfrac{2}{3}m_a\\
GB&=\dfrac{2}{3}m_b\\
\shortintertext{y}
GC&=\dfrac{2}{3}m_c\\
{GA}^2+{GB}^2+{GC}^2&=\dfrac{4}{9}({m_a}^2+{m_b}^2+{m_c}^2)\\
&\underset{\uparrow}{=}\dfrac{4}{9}\cdot\dfrac{3}{4}(a^2+b^2+c^2)\\
&\text{Prop. \eqref{181}}\\
&=\dfrac{1}{3}(a^2+b^2+c^2)
\end{align*}
\end{proof}
Ahora si demostremos la segunda parte del problema \eqref{208}:\\[.3cm]<<$\forall X$\, del plano: ${XA}^2+{XB}^2+{XC}^2=\underset{\text{cte}}{\underbrace{({GA}^2+{GB}^2+{GC}^2)}}+3{XG}^2$>>, (véase Fig. \ref{209}).
\begin{figure}[ht!]
\begin{center}
\includegraphics[scale=0.3]{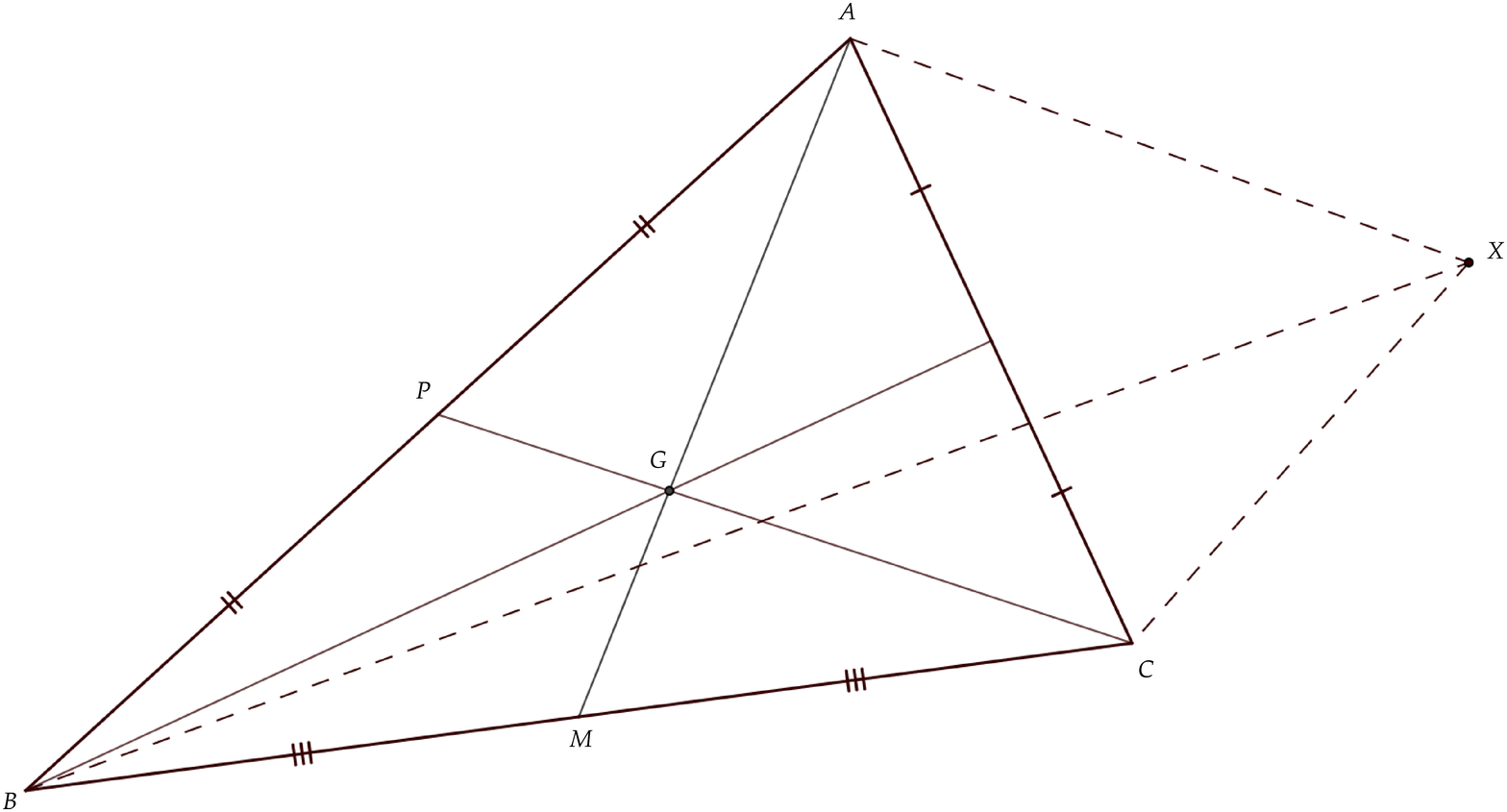}\\
\caption{}\label{209}
\end{center}
\end{figure}
\begin{proof}
\noindent
\begin{enumerate}
\item[$\bullet$] Consideremos la mediana $\overline{AM}$.\\[.3cm] Sea $D$ el punto medio de $\overline{AG}$, (véase Fig. \ref{210}).\\[.3cm] Por el Teorema de la mediana en el triángulo: $$\triangle XBC: {XB}^2+{XC}^2=2{XM}^2+\dfrac{a^2}{2}$$
\begin{figure}[ht!]
\begin{center}
\includegraphics[scale=0.3]{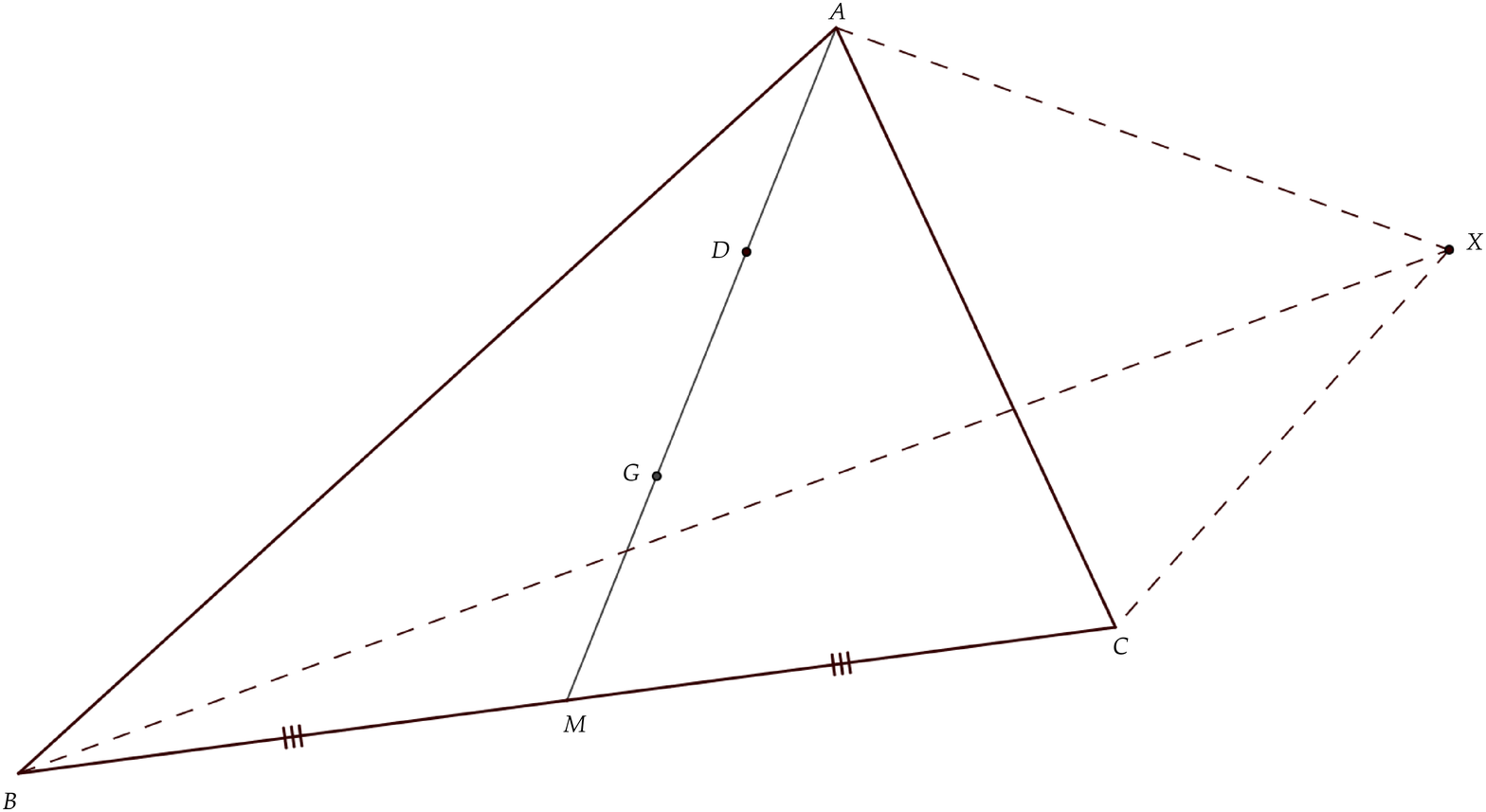}\\
\caption{}\label{210}
\end{center}
\end{figure}
\begin{align*}
&\text{En el $\triangle XDM: {XD}^2+{XM}^2=2{XG}^2+\dfrac{1}{2}{DM}^2$}\\
&\text{En el $\triangle XAG: {XA}^2+{XG}^2=2{XD}^2+\dfrac{1}{2}{AG}^2$}
\end{align*}
\noindent Multiplicando la segunda ecuación por 2 y añadiendo:
\begin{align*}
{XA}^2+{XB}^2+{XC}^2-3{XG}^2&=\dfrac{a^2}{2}+{DM}^2+\dfrac{1}{2}{AG}^2\\
&\underset{\uparrow}{=}\dfrac{a^2}{2}+\dfrac{3}{2}{AG}^2\\
&DM=AG
\end{align*}
\noindent O sea que,
\begin{equation}\label{182}
{XA}^2+{XB}^2+{XC}^2-3{XG}^2=\dfrac{a^2}{2}+\dfrac{3}{2}{AG}^2
\end{equation}
\item[$\bullet$] Consideremos la mediana de $\overline{BG}$.\\[.3cm] Llamemos $E$ el punto medio de $\overline{BG}$, (véase Fig. \ref{211}).\\[.3cm]
\begin{figure}[ht!]
\begin{center}
\includegraphics[scale=0.3]{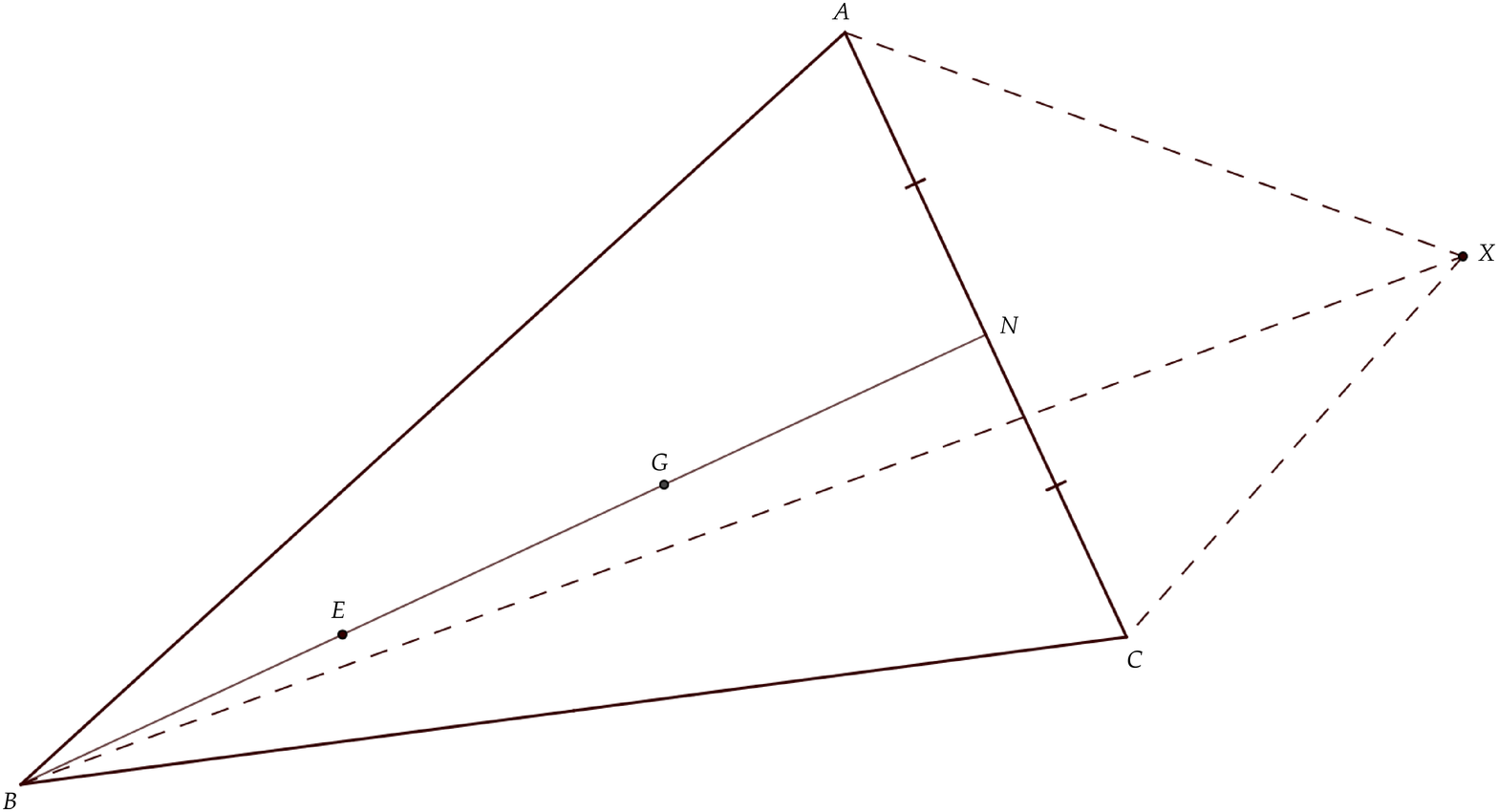}\\
\caption{}\label{211}
\end{center}
\end{figure}
\newline
Aplicando el Teorema de la mediana al triángulo $\triangle XAC$:
\begin{align*}
\shortintertext{Al $\triangle XAC$}
{XA}^2+{XC}^2&=2{XN}^2+\dfrac{b^2}{2}\\
\shortintertext{En el $\triangle XNC$:}
{XN}^2+{XE}^2&=2{XG}^2+\dfrac{{EN}^2}{2}\\
\shortintertext{En el $\triangle XGB$:}
{XG}^2+{XB}^2&=2{XE}^2+\dfrac{{BG}^2}{2}
\end{align*}
Multiplicando la segunda ecuación por 2 y sumando:
\begin{align*}
{XA}^2+{XB}^2+{XC}^2-3{XG}^2&=\dfrac{b^2}{2}+{EN}^2+\dfrac{{BG}^2}{2}\\
&\underset{\uparrow}{=}\dfrac{b^2}{2}+\dfrac{3}{2}{BG}^2\\
&=EN=BG
\end{align*}
O sea que,
\begin{equation}\label{183}
{XA}^2+{XB}^2+{XC}^2-3{XG}^2=\dfrac{b^2}{2}+\dfrac{3}{2}{BG}^2
\end{equation}
\item[$\bullet$] Consideremos ahora la mediana de $\overline{CP}$.\\[.3cm] Llamemos $F$ el punto medio de $\overline{CG}$, (véase Fig. \ref{212}).\\[.3cm]
\begin{figure}[ht!]
\begin{center}
\includegraphics[scale=0.3]{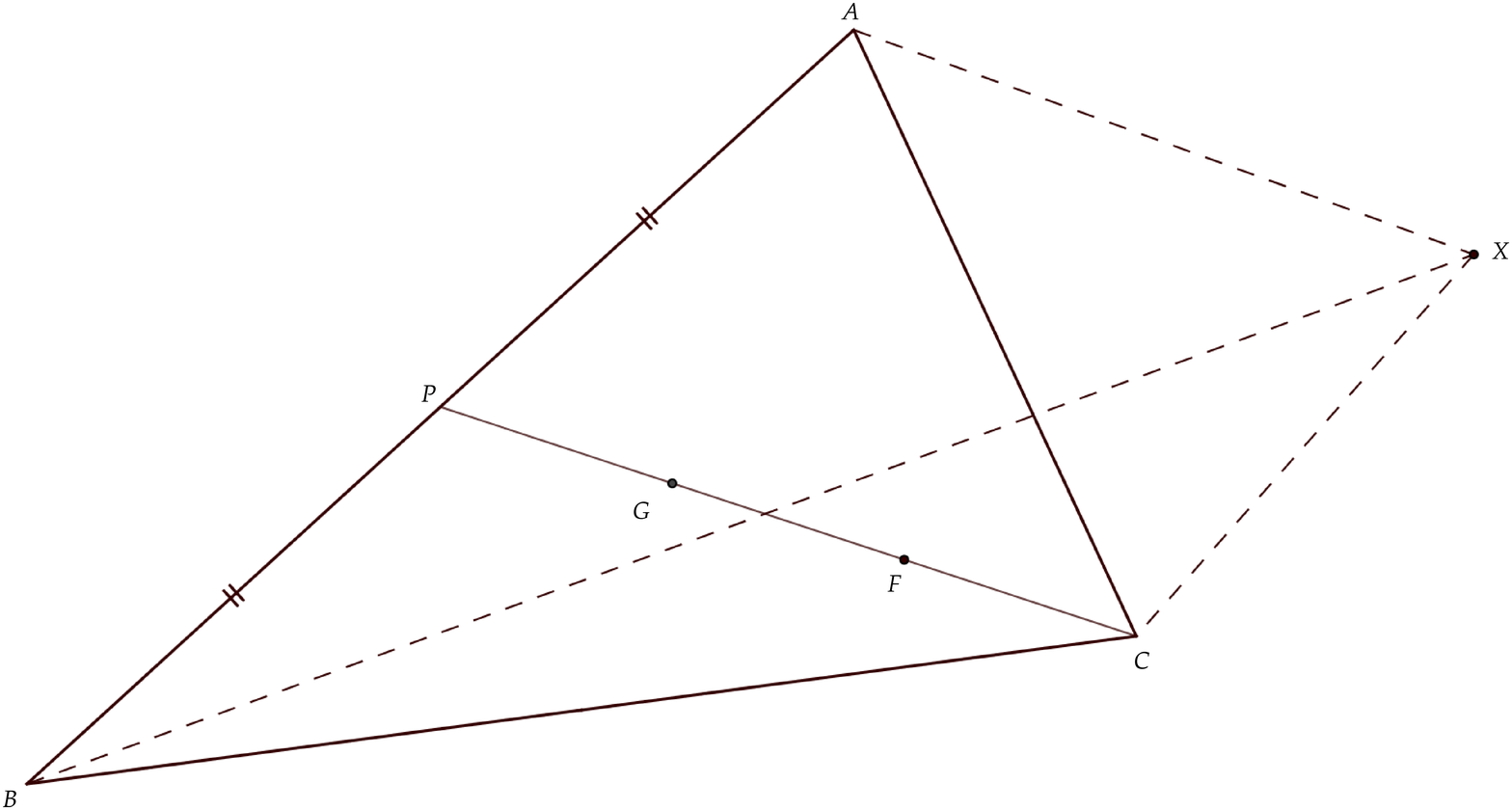}\\
\caption{}\label{212}
\end{center}
\end{figure}
\newline
Al aplicar el Teorema de la mediana al triángulo $\triangle XAB$:
\begin{align*}
{XA}^2+{XB}^2&=2{XP}^2+\dfrac{c^2}{2}\\
\shortintertext{En el $\triangle XPF$:}
{XP}^2+{XF}^2&=2{XG}^2+\dfrac{{AF}^2}{2}\\
\shortintertext{En el $\triangle XGC$:}
{XG}^2+{XC}^2&=2{XF}^2+\dfrac{{GC}^2}{2}
\end{align*}
Multiplicando la segunda ecuación por dos y sumando:
\begin{align*}
{XA}^2+{XB}^2+{XC}^2-3{XG}^2&=\dfrac{c^2}{2}+{PF}^2+\dfrac{{GC}^2}{2}\\
&\underset{\uparrow}{=}\dfrac{c^2}{2}+\dfrac{3}{2}{GC}^2\\
&PF=GC
\end{align*}
O sea que,
\begin{equation}\label{184}
{XA}^2+{XB}^2+{XC}^2-3{XG}^2=\dfrac{c^2}{2}+\dfrac{3}{2}{GC}^2
\end{equation}
Al sumar \eqref{182}, \eqref{183} y \eqref{184}:
\begin{align*}
3\left({XA}^2+{XB}^2+{XC}^2\right)-9{XG}^2&=\dfrac{1}{2}\left(a^2+b^2+c^2\right)+\dfrac{3}{2}\left({GA}^2+{GB}^2+{GC}^2\right)\\
&=\dfrac{1}{2}3\left({GA}^2+{GB}^2+{GC}^2\right)+\dfrac{3}{2}\left({GA}^2+{GB}^2+{GC}^2\right)\\
&=\dfrac{6}{2}\left({GA}^2+{GB}^2+{GC}^2\right)\\
&=3\left({GA}^2+{GB}^2+{GC}^2\right)\\
\shortintertext{Luego,}
{XA}^2+{XB}^2+{XC}^2&={GA}^2+{GB}^2+{GC}^2+3{XG}^2\\
\shortintertext{o también,}
{XA}^2+{XB}^2+{XC}^2&=\dfrac{1}{3}(a^2+b^2+c^2)+3{XG}^2
\end{align*}
\end{enumerate}
\end{proof}
\begin{corol}
Si $X$ se toma en $G, XG=0$ y se tiene que,
\begin{align*}
\min\left\{{XA}^2+{XB}^2+{XC}^2\right\}&={GA}^2+{GB}^2+{GC}^2\\
&=\dfrac{1}{3}(a^2+b^2+c^2)
\end{align*}
y se alcanza en el punto $G\hspace{1.5cm}\blacksquare$\\[.3cm]
\end{corol}
\begin{corol}
Si $X$ se toma en $O$: el circuncentro del triángulo $\triangle ABC$,\\
\begin{align*}
{OA}^2+{OB}^2+{OC}^2&={GA}^2+{GB}^2+{GC}^2+3{OG}^2\\
\shortintertext{Pero,}
OA&=OB=OC=R:\text{radio de la circunferencia circunscrita al triángulo}\\
\shortintertext{Luego,}
3R^2&=\dfrac{1}{3}(a^2+b^2+c^2)+3{OG}^2\\
\shortintertext{y por lo tanto,}
{OG}^2&=R^2-\dfrac{1}{9}(a^2+b^2+c^2)\hspace{1.5cm}\blacksquare
\end{align*}
\end{corol}
\newpage
\thispagestyle{empty}